\documentclass[11pt]{article}       % onecolumn (second format)

\usepackage{graphicx}
\usepackage{latexsym}
\usepackage{amsfonts}
\usepackage{amsmath}
\usepackage{amsthm}
\usepackage{bm}
\usepackage{latexsym} \usepackage{epsfig}

\newtheorem*{theorem*}{Theorem}
\newtheorem{theorem}{Theorem}[section]
\newtheorem{lemma}[theorem]{Lemma}
\newtheorem*{proposition*}{Proposition}
\newtheorem{maintheorem}{Theorem}

\newtheorem{corollary}[theorem]{Corollary}

\newtheorem{definition}[theorem]{Definition}
\newtheorem{remark}[theorem]{Remark}

\newtheorem{example}[theorem]{Example}
\newtheorem{claim}[theorem]{Claim}

\renewcommand{\P}[1]{{\mathbb{P}}\left[{#1}\right]}
\newcommand{\CondP}[2]{{\mathbb{P}}\left[{#1}\middle\vert{#2}\right]}
\newcommand{\E}[1]{{\mathbb{E}}\left[{#1}\right]}
\newcommand{\ind}[1]{{\bf 1}\left({#1}\right)}

\newcommand{\eps}{\epsilon}

\newcommand{\N}{\mathbb N}

\newcommand{\cF}{{\mathcal{F}}}

\newcommand{\cG}{{\mathcal{G}}}
\newcommand{\cK}{{\mathcal{K}}}

\newcommand{\half}{{\textstyle \frac12}}

\newcommand{\argmax}{\operatornamewithlimits{argmax}}

\renewcommand{\phi}{\varphi}

\newcommand{\InfGraphs}{\mathcal{B}}
\newcommand{\GraphFam}{\mathcal{A}}

\newcommand{\belief}{X}
\newcommand{\limac}{A}
\newcommand{\llr}{Z} %private log likelihood ratio
\newcommand{\psignal}{W} %private signal
\newcommand{\action}{A}
\newcommand{\estL}{K}
\newcommand{\dtv}{d_{TV}}

\begin{document}
\title{Asymptotic Learning on Bayesian Social Networks\thanks{The
    authors would like to thank Shachar Kariv for an enthusiastic
    introduction to his work with Douglas Gale, and for suggesting the
    significance of asymptotic learning in this model.
    Elchanan Mossel is supported by NSF NSF award DMS 1106999, by ONR
    award N000141110140 and by ISF grant 1300/08. Allan Sly is
    supported in part by an Alfred Sloan Fellowship in
    Mathematics. Omer Tamuz is supported by ISF grant 1300/08, and is
    a recipient of the Google Europe Fellowship in Social Computing.
    This research is supported in part by this Google Fellowship.}}

\author{Elchanan Mossel \and Allan Sly \and Omer Tamuz}

\maketitle
\begin{abstract}
  Understanding information exchange and aggregation on networks is a
  central problem in theoretical economics, probability and
  statistics.  We study a standard model of economic agents on the
  nodes of a social network graph who learn a binary ``state of the
  world'' $S$, from initial signals, by repeatedly observing each
  other's best guesses.

  Asymptotic learning is said to occur on a family of graphs $G_n =
  (V_n,E_n)$ with $|V_n| \to \infty$ if with probability tending to
  $1$ as $n \to \infty$ all agents in $G_n$ eventually estimate $S$
  correctly.  We identify sufficient conditions for asymptotic
  learning and contruct examples where learning does not occur when
  the conditions do not hold.

\end{abstract}

\section{Introduction}
We consider a directed graph $G$ representing a social network. The
nodes of the graph are the set of agents $V$, and an edge from agent
$u$ to $w$ indicates that $u$ can observe the actions of $w$.  The
agents try to estimate a binary {\em state of the world} $S \in
\{0,1\}$, where each of the two possible states occurs with
probability one half.

The agents are initially provided with private signals which are
informative with respect to $S$ and i.i.d., conditioned on $S$:
There are two distributions, $\mu_0 \neq \mu_1$, such that conditioned
on $S$, the private signals are independent and distributed $\mu_S$.

In each time period $t \in \N$, each agent $v$ chooses an ``action''
$\action_v(t)$, which equals whichever of $\{0,1\}$ the state of the
world is more likely to equal, conditioned on the information
available to $v$ at time $t$. This information includes its private
signal, as well as the actions of its social network neighbors in the
previous periods.

A first natural question is whether the agents eventually reach
consensus, or whether it is possible that neighbors ``agree to
disagree'' and converge to different actions.  Assuming that the
agents do reach consensus regarding their estimate of $S$, a second
natural question is whether this consensus estimator is equal to $S$.
Certainly, since private signals are independent conditioned on $S$, a
large enough group of agents has, in the aggregation of their private
signals, enough information to learn $S$ with high
probability. However, it may be the case that this information is not
disseminated by the above described process. These and related
questions have been studied extensively in economics, statistics and
operations research; see Section~\ref{sec:literature}.

We say that the agents learn on a social network graph $G$ when all
their actions converge to the state of the world $S$.  For a sequence
of graphs $\{G_n\}_{n=1}^\infty$ such that $G_n$ has $n$ agents, we
say that {\em Asymptotic learning} occurs when the probability that
the agents learn on $G_n$ tends to one as $n$ tends to infinity, for a
fixed choice of private signal distributions $\mu_1$ and $\mu_0$.

An agent's initial {\em private belief} is the probability that $S=1$,
conditioned only on its private signal. When the distribution of
private beliefs is atomic, asymptotic learning does not necessarily
occur (see Example~\ref{example:non-learning}). This is also the case
when the social network graph is undirected (see
Example~\ref{example:directed}). Our main result
(Theorem~\ref{thm:non-atomic-learning}) is that asymptotic learning
occurs for non-atomic private beliefs and undirected graphs.

To prove this theorem we first prove that the condition of non-atomic
initial private beliefs implies that the agents all converge to the
same action, or all don't converge at all
(Theorem~\ref{thm:unbounded-common-knowledge}). We then show that for
any model in which this holds, asymptotic learning occurs
(Theorem~\ref{thm:bounded-learning}).  Note that it has been shown
that agents reach agreement under various other conditions (cf.
M{\'e}nager~\cite{menager2006consensus}). Hence, by
Theorem~\ref{thm:bounded-learning}, asymptotic learning also holds for
these models.

Our proof includes several novel insights into the dynamics of
interacting Bayesian agents.  Broadly, we show that on undirected
social network graphs connecting a {\em countably infinite} number of
agents, if all agents converge to the same action then they converge
to the correct action. This follows from the observation that if
agents in distant parts of a large graph converge to the same action
then they do so {\em almost} independently.  We then show that this
implies that for {\em finite} graphs of growing size the probability
of learning approaches one.

At its heart of this proof lies a topological lemma
(Lemma~\ref{lemma:local_property}) which may be of independent
interest; the topology here is one of rooted graphs (see, e.g.,
Benjamini and Schramm~\cite{benjamini2011recurrence}, Aldous and
Steele~\cite{aldous2003objective}). The fact that asymptotic learning
occurs for undirected graphs (as opposed to general strongly connected
graphs) is related to the fact that sets of bounded degree, {\it
  undirected} graphs are compact in this topology. In fact, our proof
applies equally to any such compact sets. For example, one can replace
{\it undirected} with {\it $L$-locally strongly connected}: a directed
graph $G=(V,E)$ is $L$-locally strongly connected if, for each $(u,w)
\in E$, there exists a path in $G$ of length at most $L$ from $w$ to
$u$. Asymptotic learning also takes place on $L$-locally strongly
connected graphs, for fixed $L$, since sets of $L$-locally strongly
connected, uniformly bounded degree graphs are compact. See
Section~\ref{sec:locally-connected} for further discussion.

\subsection{Related literature}
\label{sec:literature}
\subsubsection{Agreement}
There is a vast economic literature studying the question of
convergence to consensus in dynamic processes and games. A founding
work is Aumann's seminal Agreement Theorem~\cite{aumann1976agreeing},
which states that Bayesian agents who observe beliefs (i.e., posterior
probabilities, as opposed to actions in our model) cannot ``agree to
disagree''. Subsequent work (notably Geanakoplos and
Polemarchakis~\cite{geanakoplos1982we}, Parikh and
Krasucki~\cite{parikh1990communication}, McKelvey and
Page~\cite{mckelvey1986common}, Gale and Kariv~\cite{GaleKariv:03}
M{\'e}nager~\cite{menager2006consensus} and Rosenberg, Solan and
Vieille~\cite{rosenberg2009informational}) expanded the range of
models that display convergence to consensus. One is, in fact, left
with the impression that it takes a pathological model to feature
interacting Bayesian agents who do ``agree to disagree''.

M{\'e}nager~\cite{menager2006consensus} in particular describes a
model similar to ours and proves that consensus is achieved in a
social network setting under the condition that the probability space
is finite and ties cannot occur (i.e., posterior beliefs are always
different than one half). Note that our asymptotic learning result
applies for any model where consensus is guaranteed, and hence in
particular applies to models satisfying M{\'e}nager's conditions.

\subsubsection{Agents on social networks}
Gale and Kariv~\cite{GaleKariv:03} also consider Bayesian agents who
observe each other's actions. They introduce a model in which, as in
ours, agents receive a single initial private signal, and the action
space is discrete. However, there is no ``state of the world'' or
conditionally i.i.d.\ private signals. Instead, the relative merit of
each possible action depends on all the private signals.  Our model is
in fact a particular case of their model, where we restrict our
attention to the particular structure of the private signals described
above.

Gale and Kariv show (loosely speaking) that neighboring agents who
converge to two different actions must, at the limit, be indifferent
with respect to the choice between these two actions. Their result is
therefore also an agreement result, and makes no statement on the
optimality of the chosen actions, although they do profess interest in
the question of {\em ``... whether the common action chosen
  asymptotically is optimal, in the sense that the same action would
  be chosen if all the signals were public information... there is no
  reason why this should be the case.''} This is precisely the
question we address.

A different line of work is the one explored by Ellison and
Fudenberg~\cite{EllFud:95}. They study agents on a social network that
use rules of thumb rather than full Bayesian updates.  A similar
approach is taken by Bala and Goyal~\cite{BalaGoyal:96}, who also
study agents acting iteratively on a social network. They too are
interested in asymptotic learning (or ``complete learning'', in their
terms). They consider a model of bounded rationality which is not
completely Bayesian. One of their main reasons for doing so is the
mathematical complexity of the fully Bayesian model, or as they state,
{\em ``to keep the model mathematically tractable... this possibility
  [fully Bayesian agents] is precluded in our model... simplifying the
  belief revision process considerably.''}  In this simpler,
non-Bayesian model, Bala and Goyal show both behaviors of asymptotic
learning and results of non-learning, depending on various parameters
of their model.

\subsubsection{Herd behavior}
The ``herd behavior'' literature (cf. Banerjee~\cite{Banerjee:92},
Bikhchandani, Hirshleifer and Welch~\cite{BichHirshWelch:92}, Smith
and S{\o}rensen~\cite{smith2000pathological}) consider related but
fundamentally simpler models. As in our model there is a ``state of
the world'' and conditionally independent private signals. A countably
infinite group of agents is exogenously ordered, and each picks an
action sequentially, after observing the actions of its predecessors
or some of its predecessors. Agents here act only {\em once}, as
opposed to our model in which they act {\em repeatedly}.

The main result for these models is that in some situations there may
arise an ``information cascade'', where, with positive probability,
almost all the agents take the wrong action. This is precisely the
opposite of {\em asymptotic learning}. The condition for information
cascades is ``bounded private beliefs''; herd behavior occurs when the
agents' beliefs, as inspired by their private signals, are bounded
away both from zero and from one~\cite{smith2000pathological}. In
contrast, we show that in our model asymptotic learning occurs even
for bounded beliefs.

In the herd behavior models information only flows in one direction:
If agent $u$ learns from $w$ then $w$ does not learn from $u$. This
significant difference, among others, makes the tools used for their
analysis irrelevant for our purposes.

\section{Formal definitions, results and examples}
\label{sec:formal}

\subsection{Main definitions}
The following definition of the agents, the state of the world and the
private signals is adapted from~\cite{MosselTamuz10B:arxiv}, where a
similar model is discussed.
\begin{definition}
  \label{def:state-signals}
  Let $(\Omega, \mathcal{O})$ be a $\sigma$-algebra. Let $\mu_0$ and
  $\mu_1$ be different and mutually absolutely continuous probability
  measures on $(\Omega, \mathcal{O})$.

  Let $\delta_0$ and $\delta_1$ be the distributions on $\{0,1\}$ such
  that $\delta_0(0)=\delta_1(1)=1$.

  Let $V$ be a countable (finite or infinite) set of agents, and let
  \begin{align*}
   \mathbb{P} = \half\delta_0\mu_0^V+\half\delta_1\mu_1^V,
  \end{align*}
  be a distribution over $\{0,1\}\times\Omega^V$.  We denote by $S
  \in\{0,1\}$ the {\bf state of the world} and by $\psignal_u$ the
  {\bf private signal} of agent $u \in V$. Let
  \begin{align*}
    (S,\psignal_{u_1},\psignal_{u_2},\ldots) \sim \mathbb{P}.
  \end{align*}
\end{definition}
Note that the private signals $\psignal_u$ are i.i.d., conditioned on
$S$: if $S=0$ - which happens with probability half - the private
signals are distributed i.i.d.\ $\mu_0$, and if $S=1$ then they are
distributed i.i.d.\ $\mu_1$.

We now define the dynamics of the model.
\begin{definition}
  \label{def:revealed-actions}
  Consider a set of agents $V$, a state of the world $S$ and private
  signals $\{\psignal_u: u\in V\}$ such that
  \begin{align*}
    (S,\psignal_{u_1},\psignal_{u_2},\ldots) \sim \mathbb{P},
  \end{align*}
  as defined in Definition~\ref{def:state-signals}.

  Let $G=(V,E)$ be a directed graph which we shall call the {\bf
    social network}. We assume throughout that $G$ is simple (i.e., no
  parallel edges or loops) and strongly connected. Let the set of
  neighbors of $u$ be $N(u) = \{v :\: (u,v) \in E\}$. The {\bf
    out-degree} of $u$ is equal to $|N(u)|$.

  For each time period $t \in \{1, 2, \ldots\}$ and agent $u \in V$,
  denote the {\bf action} of agent $u$ at time $t$ by $\action_u(t)$,
  and denote by $\cF_u(t)$ the information available to agent $u$ at
  time $t$. They are jointly defined by
  \begin{align*}
    \cF_u(t) = \sigma(\psignal_u,\{\action_v(t'):\:v \in N(u), t' < t\}),
  \end{align*}
  and
  \begin{align*}
    \action_u(t) =
    \begin{cases}
      0 & \CondP{S=1}{\cF_u(t)} < 1/2 \\
      1 & \CondP{S=1}{\cF_u(t)} > 1/2 \\
      \in \{0,1\} & \CondP{S=1}{\cF_u(t)} = 1/2.
    \end{cases}
  \end{align*}
  Let $\belief_u(t) = \CondP{S=1}{\cF_u(t)}$ be agent $u$'s {\bf
    belief} at time $t$.
\end{definition}
Informally stated, $\action_u(t)$ is agent $u$'s best estimate of $S$
given the information $\cF_u(t)$ available to it up to time $t$. The
information available to it is its private signal $\psignal_u$ and the
actions of its neighbors in $G$ in the previous time periods.

\begin{remark}
  \label{rem:map}
  An alternative and equivalent definition of $\action_u(t)$ is the
  MAP estimator of $S$, as calculated by agent $u$ at time $t$:
  \begin{align*}
    \action_u(t)=\argmax_{s \in \{0,1\}}\CondP{S=s}{\cF(t)}
    =\argmax_{A \in \cF(t)}\P{A=S},
  \end{align*}
  with some tie-breaking rule.
\end{remark}

Note that we assume nothing about how agents break ties, i.e., how
they choose their action when, conditioned on their available
information, there is equal probability for $S$ to equal either 0 or 1.

Note also that the belief of agent $u$ at time $t=1$, $\belief_u(1)$,
depends only on $\psignal_u$:
\begin{align*}
  \belief_u(1) = \CondP{S=1}{\psignal_u}.
\end{align*}
We call $\belief_u(1)$ the {\em initial belief} of agent $u$.
\begin{definition}
  Let $\mu_0$ and $\mu_1$ be such that $X_u(1)$, the initial belief of
  $u$, has a non-atomic distribution ($\Leftrightarrow$ the
  distributions of the initial beliefs of all agents are
  non-atomic). Then we say that the pair $(\mu_0,\mu_1)$ {\bf induce
    non-atomic beliefs}.
\end{definition}

We next define some limiting random variables: $\cF_u$ is the limiting
information available to $u$, and $\belief_u$ is its limiting belief.
\begin{definition}
  Denote $\cF_u = \cup_t\cF_u(t)$, and let
  \begin{align*}
    \belief_u = \CondP{S=1}{\cF_u}.
  \end{align*}
\end{definition}
Note that the limit $\lim_{t \to \infty}\belief_u(t)$ almost surely
exists and equals $\belief_u$, since $\belief_u(t)$ is a bounded
martingale.

We would like to define the limiting {\em action} of agent
$u$. However, it might be the case that agent $u$ takes both actions
infinitely often, or that otherwise, at the limit, both actions are
equally desirable. We therefore define $\limac_u$ to be the limiting
{\em optimal action set}. It can take the values $\{0\}$, $\{1\}$ or
$\{0,1\}$.
\begin{definition}
  Let $\limac_u$, the {\bf optimal action set} of agent $u$, be defined by
  \begin{align*}
    \limac_u =
    \begin{cases}
      \{0\} & \belief_u < 1/2 \\
      \{1\} & \belief_u > 1/2 \\
      \{0,1\} & \belief_u = 1/2.
    \end{cases}
  \end{align*}
\end{definition}
Note that if $a$ is an action that $u$ takes infinitely often then $a
\in \limac_u$, but that if $0$ (say) is the only action that $u$ takes
infinitely often then it still may be the case that
$\limac_u=\{0,1\}$. However, we show below that when $(\mu_0,\mu_1)$
induce non-atomic beliefs then $\limac_u$ is almost surely equal to
the set of actions that $u$ takes infinitely often.

\subsection{Main results}
In our first theorem we show that when initial private beliefs are
non-atomic, then at the limit $t \to \infty$ the optimal action sets
of the players are identical. As Example~\ref{example:non-learning}
indicates, this may not hold when private beliefs are atomic.

\begin{maintheorem}
  \label{thm:unbounded-common-knowledge}
  Let $(\mu_0,\mu_1)$ induce non-atomic beliefs. Then there exists
  a random variable $\limac$ such that almost surely $\limac_u=\limac$
  for all $u$.
\end{maintheorem}
I.e., when initial private beliefs are non-atomic then agents, at the
limit, agree on the optimal action. The following theorem states that
when such agreement is guaranteed then the agents learn the state of
the world with high probability, when the number of agents is
large. This phenomenon is known as {\em asymptotic learning}. This
theorem is our main result.
\begin{maintheorem}
  \label{thm:bounded-learning}
  Let $\mu_0,\mu_1$ be such that for every connected, undirected graph
  $G$ there exists a random variable $\limac$ such that almost surely
  $\limac_u=\limac$ for all $u \in V$.  Then there exists a sequence
  $q(n)=q(n, \mu_0, \mu_1)$ such that $q(n) \to 1$ as $n \to \infty$,
  and $\P{\limac = \{S\}} \geq q(n)$, for any choice of undirected,
  connected graph $G$ with $n$ agents.
\end{maintheorem}
Informally, when agents agree on optimal action sets then they
necessarily learn the correct state of the world, with probability
that approaches one as the number of agents grows. This holds
uniformly over all possible connected and undirected social network
graphs.

The following theorem is a direct consequence of the two theorems
above, since the property proved by
Theorem~\ref{thm:unbounded-common-knowledge} is the condition required
by Theorem~\ref{thm:bounded-learning}.
\begin{maintheorem}
  \label{thm:non-atomic-learning}
  Let $\mu_0$ and $\mu_1$ induce non-atomic beliefs. Then there exists
  a sequence $q(n)=q(n, \mu_0, \mu_1)$ such that $q(n) \to 1$ as $n
  \to \infty$, and $\P{\limac_u = \{S\}} \geq q(n)$, for all agents
  $u$ and for any choice of undirected, connected $G$ with $n$ agents.
\end{maintheorem}

\subsection{Note on directed vs. undirected graphs}

\begin{figure}[h]
  \centering
  \includegraphics[scale=1]{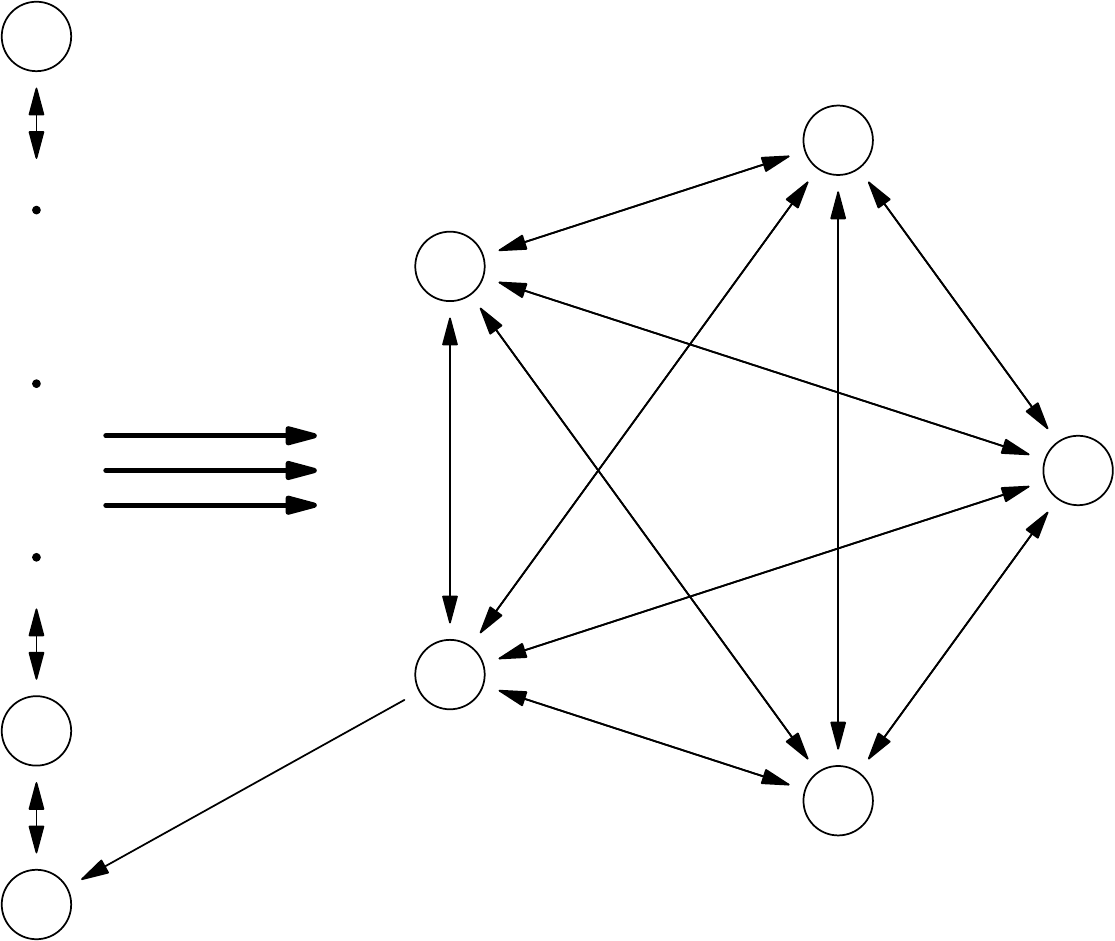}
  \caption{\label{fig:royal-family} The five members of the royal
    family (on the right) all observe each other. The rest of the
    agents - the public - all observe the royal family (as
    suggested by the three thick arrows in the middle) and their
    immediate neighbors. Finally, one of the royals observes one
    of the public, so that the graph is strongly connected.  This is
    an example of how asymptotic learning does not necessarily occur
    when the graph is undirected.}
\end{figure}

Note that we require that the graph $G$ not only be strongly
connected, but also undirected (so that if $(u,v) \in E$ then $(v,u) \in
E$.)  The following example (depicted in Figure~\ref{fig:royal-family})
shows that when private beliefs are bounded then asymptotic learning
may not occur when the graph is strongly connected but not
undirected\footnote{We draw on Bala and Goyal's~\cite{BalaGoyal:96}
  {\em royal family} graph.}.

\begin{example}
  \label{example:directed}
  Consider the the following graph. The vertex set is comprised of two
  groups of agents: a ``royal family'' clique of 5 agents who all
  observe each other, and $5-n$ agents - the ``public'' - who are
  connected in a chain, and in addition can all observe all the agents
  in the royal family. Finally, a single member of the royal family
  observes one of the public, so that the graph is strongly connected.
\end{example}

Now, with positive probability, which is independent of $n$, there
occurs the event that all the members of the royal family initially
take the wrong action. Assuming the private signals are sufficiently
weak, then it is clear that all the agents of the public will adopt
the wrong opinion of the royal family and will henceforth choose the
wrong action.

Note that the removal of one edge - the one from the royal back to the
commoners - results in this graph no longer being strongly
connected. However, the information added by this edge rarely has an
affect on the final outcome of the process. This indicates that strong
connectedness is too weak a notion of connectedness in this
context. We therefore in seek stronger notions such as connectedness
in undirected graphs.

A weaker notion of connectedness is that of $L$-locally strongly
connected graphs, which we defined above. For any $L$, the graph from Example~\ref{example:directed} is not $L$-locally strongly
connected for $n$ large enough.

\section{Proofs}
Before delving into the proofs of
Theorems~\ref{thm:unbounded-common-knowledge}
and~\ref{thm:bounded-learning} we introduce additional definitions in
subsection~\ref{sec:more-defs} and prove some general lemmas in
subsections~\ref{sec:graph-limits}, \ref{sec:isomorphics-balls}
and~\ref{sec:delta-ind}.  Note that Lemma~\ref{lemma:local_property},
which is the main technical insight in the proof of
Theorem~\ref{thm:bounded-learning}, may be of independent interest. We
prove Theorem~\ref{thm:bounded-learning} in
subsection~\ref{sec:learning} and
Theorem~\ref{thm:unbounded-common-knowledge} in
subsection~\ref{sec:convergence}.

\subsection{Additional general notation}
\label{sec:more-defs}

\begin{definition}
  \label{def:llr}
  We denote the {\bf log-likelihood ratio} of agent $u$'s belief at
  time $t$ by
  \begin{align*}
    \llr_u(t)  =\log\frac{\belief_u(t)}{1-\belief_u(t)},
  \end{align*}
  and let
  \begin{align*}
    \llr_u = \lim_{t\to\infty}\llr_u(t).
  \end{align*}
\end{definition}
Note that
\begin{align*}
  \llr_u(t)=   \log\frac{\CondP{S=1}{\cF_u(t)}}{\CondP{S=0}{\cF_u(t)}}.
\end{align*}
and that
\begin{align*}
  \llr_u(1) = \log\frac{d\mu_1}{d\mu_0}(W_u).
\end{align*}
Note also that $\llr_u(t)$ converges almost surely since
$\belief_u(t)$ does.

\begin{definition}
  \label{def:agg-actions}
  We denote the set of actions of agent $u$ up to time $t$ by
  \begin{align*}
    \bar{\action}_u(t) = (\action_u(1),\ldots,\action_u(t-1)).
  \end{align*}
  The set of all actions of $u$ is similarly denoted by
  \begin{align*}
    \bar{\action}_u = (\action_u(1),\action_u(2),\ldots).
  \end{align*}
  We denote the actions of the neighbors of $u$ up to time $t$ by
  \begin{align*}
    I_u(t) = \{\bar{\action}_w(t):\: w \in N(u)\} = \{\action_w(t'):\:w
    \in N(u),t'<t\},
  \end{align*}
  and let $I_u$ denote all the actions of $u$'s neighbors:
  \begin{align*}
    I_u = \{\bar{\action}_w:\: w \in N(u)\} = \{\action_w(t'):\:w
    \in N(u),t' \geq 1\}.
  \end{align*}
\end{definition}
Note that using this notation we have that $\cF_u(t) =
\sigma(\psignal_u,I_u(t))$ and $\cF_u = \sigma(\psignal_u,I_u)$.

\begin{definition}
  \label{def:p-u}
  We denote the probability that $u$ chooses the correct action at time
  $t$ by
  \begin{align*}
    p_u(t)=\P{\action_u(t)=S}.
  \end{align*}
  and accordingly
  \begin{align*}
    p_u=\lim_{t \to \infty}p_u(t).
  \end{align*}
\end{definition}

\begin{definition}
  \label{def:group-signal}
  For a set of vertices $U$ we denote by $\psignal(U)$ the private
  signals of the agents in $U$.
\end{definition}

\subsection{Sequences of rooted graphs and their limits}
\label{sec:graph-limits}
In this section we define a topology on {\em rooted graphs}. We call
convergence in this topology {\em convergence to local limits}, and
use it repeatedly in the proof of
Theorem~\ref{thm:bounded-learning}. The core of the proof of
Theorem~\ref{thm:bounded-learning} is the topological
Lemma~\ref{lemma:local_property}, which we prove here. This lemma is a
claim related to {\em local graph properties}, which we also introduce
here.

\begin{definition}
  Let $G=(V,E)$ be a finite or countably infinite graph, and let $u
  \in V$ be a vertex in $G$. We denote by $(G,u)$ the {\bf rooted
    graph} $G$ with root $u$.
\end{definition}

\begin{definition}
  Let $G=(V,E)$ and $G'=(V',E')$ be graphs. $h : V \to V'$ is a {\bf
    graph isomorphism} between $G$ and $G'$ if $(u,v) \in E
  \Leftrightarrow (h(u), h(v)) \in E'$.

  Let $(G,u)$ and $(G',u')$ be rooted graphs. Then $h : V \to V'$ is a
  {\bf rooted graph isomorphism} between $(G,u)$ and $(G',u')$ if $h$
  is a graph isomorphism and $h(u) = u'$.

  We write $(G,u) \cong (G',u')$ whenever there exists a rooted graph
  isomorphism between the two rooted graphs.
\end{definition}

Given a (perhaps directed) graph $G=(V,E)$ and two vertices $u, w \in
V$, the graph distance $d(u,w)$ is equal to the length in edges of a
shortest (directed) path between $u$ and $w$.
\begin{definition}
  \label{def:balls}
  We denote by $B_r(G, u)$ the ball of radius $r$ around the vertex
  $u$ in the graph $G=(V,E)$: Let $V'$ be the set of vertices $w$ such
  that $d(u,w)$ is at most $r$. Let $E' = \{(u,w)\in E:\: u,w \in
  V'\}$. Then $B_r(G, u)$ is the rooted graph with vertices $V'$,
  edges $E'$ and root $u'$.
\end{definition}

We next define a topology on strongly connected rooted graphs (or
rather on their isomorphism classes; we shall simply refer to these
classes as graphs).  A natural metric between strongly connected
rooted graphs is the following (see Benjamini and
Schramm~\cite{benjamini2011recurrence}, Aldous and
Steele~\cite{aldous2003objective}). Given $(G,u)$ and $(G',u')$, let
\begin{align*}
  D((G,u),(G',u')) = 2^{-R},
\end{align*}
where
\begin{align*}
  R = \sup \{r : B_r(G,u) \cong B_r(G',u')\}.
\end{align*}
This is indeed a metric: the triangle inequality follows immediately,
and a standard diagonalization argument is needed to show that if
$D((G,u),(G',u'))=0$ then $(G,u) \cong (G',u')$.

This metric induces a topology that will be useful to us. As usual,
the basis of this topology is the set of balls of the metric; the ball
of radius $2^{-R}$ around the {\em graph} $(G,u)$ is the set of graphs
$(G',u')$ such that $B_R(G,u) \cong B_R(G',u')$. We refer to
convergence in this topology as convergence to a {\em local limit},
and provide the following equivalent definition for it:
\begin{definition}
  Let $\{(G_r,u_r)\}_{r=1}^\infty$ be a sequence of strongly connected
  rooted graphs. We say that the sequence converges if there exists a
  strongly connected rooted graph $(G',u')$ such that
  \begin{align*}
    B_r(G',u') \cong B_r(G_r,u_r),
  \end{align*}
  for all $r \geq 1$.  We then write
  \begin{align*}
    (G',u')=\lim_{r \to \infty}(G_r,u_r),
  \end{align*}
  and call $(G',u')$ the {\bf local limit} of the sequence
  $\{(G_r,u_r)\}_{r=1}^\infty$.
\end{definition}

Let $\cG_d$ be the set of strongly connected rooted graphs with degree
at most $d$.  Another standard diagonalization argument shows that
$\cG_d$ is {\em compact} (see
again~\cite{benjamini2011recurrence,aldous2003objective}). Then, since
the space is metric, every sequence in $\cG_d$ has a converging
subsequence:
\begin{lemma}
  \label{lemma:compactness}
  Let $\{(G_r,u_r)\}_{r=1}^\infty$ be a sequence of rooted graphs in
  $\cG_d$. Then there exists a subsequence
  $\{(G_{r_i},u_{r_i})\}_{i=1}^\infty$ with $r_{i+1}>r_i$ for all $i$,
  such that $\lim_{i \to \infty}(G_{r_i},u_{r_i})$ exists.
\end{lemma}

We next define {\em local properties} of rooted graphs.
\begin{definition}
  \label{def:local-property}
  Let $P$ be property of rooted graphs or a Boolean predicate on
  rooted graphs. We write $(G,u) \in P$ if $(G,u)$ has the property,
  and $(G,u) \notin P$ otherwise.

  We say that $P$ is a {\bf local property} if, for every $(G,u) \in
  P$ there exists an $r>0$ such that if $B_r(G,u) \cong B_r(G',u')$,
  then $(G',u') \in P$. Let $r$ be such that $B_r(G,u) \cong
  B_r(G',u') \Rightarrow (G',u') \in P$. Then we say that {\bf $(G,u)$
    has property $P$ with radius $r$}, and denote $(G,u) \in P^{(r)}$.
\end{definition}
That is, if $(G,u)$ has a local property $P$ then there is some $r$
such that knowing the ball of radius $r$ around $u$ in $G$ is
sufficient to decide that $(G,u)$ has the property $P$. An alternative
name for a local property would therefore be a {\em locally decidable}
property. In our topology, local properties are nothing but {\em open
  sets}: the definition above states that if $(G,u) \in P$ then there
exists an element of the basis of the topology that includes $(G,u)$
and is also in $P$. This is a necessary and sufficient condition for
$P$ to be open.

We use this fact to prove the following lemma.
\begin{definition}
  \label{def:infgraphs}
  Let $\InfGraphs_d$ be the set of infinite, connected, undirected
  graphs of degree at most $d$, and let $\InfGraphs_d^r$ be the set of
  $\InfGraphs_d$-rooted graphs
  \begin{align*}
    \InfGraphs_d^r = \{(G,u) \,:\, G \in \InfGraphs_d, u \in G\}.
  \end{align*}
\end{definition}
\begin{lemma}
  \label{lemma:inf-graphs-closed}
  $\InfGraphs_d^r$ is compact.
\end{lemma}
\begin{proof}
  Lemma~\ref{lemma:compactness} states that $\cG_d$, the set of
  strongly connected rooted graphs of degree at most $d$, is
  compact. Since $\InfGraphs_d^r$ is a subset of $\cG_d$, it remains
  to show that $\InfGraphs_d^r$ is closed in $\cG_d$.

  The complement of $\InfGraphs_d^r$ in $\cG_d$ is the set of graphs
  in $\cG_d$ that are either finite or directed. These are both local
  properties: if $(G,u)$ is finite (or directed), then there exists a
  radius $r$ such that examining $B_r(G,u)$ is enough to determine
  that it is finite (or directed). Hence the sets of finite graphs and
  directed graphs in $\cG_d$ are open in $\cG_d$, their intersection
  is open in $\cG_d$, and their complement, $\InfGraphs_d^r$, is
  closed in $\cG_d$.
\end{proof}

We now state and prove the main lemma of this subsection. Note that
the set of graphs $\InfGraphs_d$ satisfies the conditions of this lemma.
\begin{lemma}
  \label{lemma:local_property}
  Let $\GraphFam$ be a set of infinite, strongly connected graphs, let
  $\GraphFam^r$ be the set of $\GraphFam$-rooted graphs
  \begin{align*}
     \GraphFam^r = \{(G,u) \,:\, G \in \GraphFam, u \in G\},
  \end{align*}
  and assume that $\GraphFam$ is such that $\GraphFam^r$ is compact.

  Let $P$ be a local property such that for each $G \in \GraphFam$
  there exists a vertex $w \in G$ such that $(G,w) \in P$. Then for
  each $G \in \GraphFam$ there exist an $r_0$ and infinitely many
  distinct vertices $\{w_n\}_{n = 1}^\infty$ such that $(G,w_n) \in
  P^{(r_0)}$ for all $n$.
\end{lemma}
\begin{figure}[h]
  \centering
  \includegraphics[scale=1]{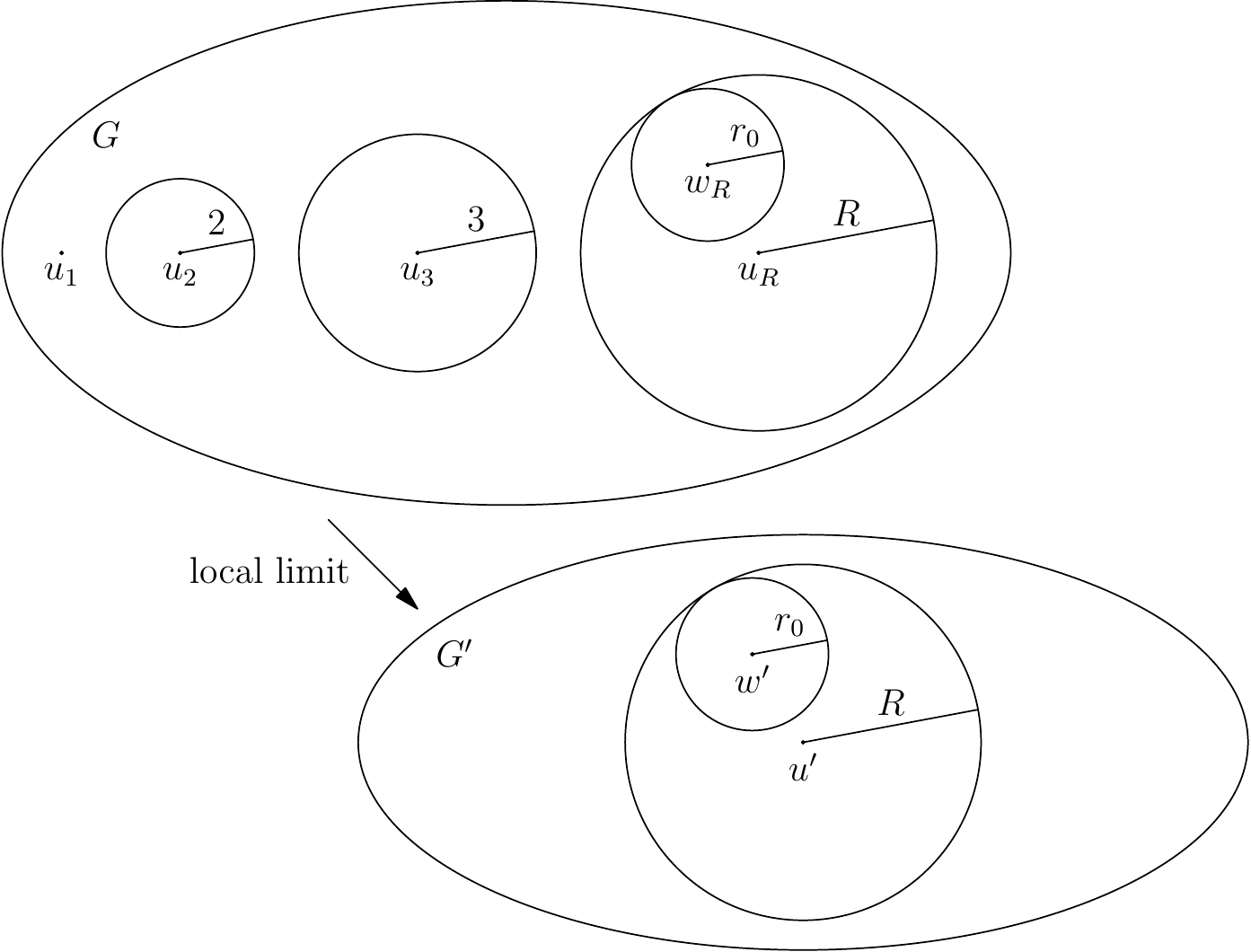}
  \caption{\label{fig:local_property} Schematic diagram of the proof
    of lemma~\ref{lemma:local_property}. The rooted graph $(G',u')$ is
    a local limit of $(G,u_r)$. For $r \geq R$, the ball $B_R(G',u')$
    is isomorphic to the ball $B_R(G,u_r)$, with $w' \in G'$
    corresponding to $w_r \in G$.}
\end{figure}

\begin{proof}
  Let $G$ be an arbitrary graph in $\GraphFam$. Consider a sequence
  $\{v_r\}_{r=1}^\infty$ of vertices in $G$ such that for all $r,s
  \in \N$ the balls $B_r(G,v_r)$ and $B_s(G,v_s)$ are disjoint.

  Since $\GraphFam^r$ is compact, the sequence $\{(G,
  v_r)\}_{r=1}^\infty$ has a converging subsequence $\{(G,
  v_{r_i})\}_{r=1}^\infty$ with $r_{i+1}>r_i$.  Write $u_r = v_{r_i}$,
  and let
  \begin{align*}
    (G',u') = \lim_{r \to \infty}(G, u_r).
  \end{align*}
  Note that since $\GraphFam^r$ is compact, $(G',u') \in \GraphFam^r$
  and in particular $G' \in \GraphFam$ is an infinite, strongly
  connected graph.  Note also that since $r_{i+1}>r_i$, it also holds
  that the balls $B_r(G,u_r)$ and $B_s(G,u_s)$ are disjoint for all
  $r,s \in \N$.

  Since $G' \in \GraphFam$, there exists a vertex $w' \in G'$ such that
  $(G',w') \in P$.  Since $P$ is a local property, $(G',w') \in
  P^{(r_0)}$ for some $r_0$, so that if $B_{r_0}(G',w') \cong
  B_{r_0}(G,w)$ then $(G,w) \in P$.

  Let $R = d(u',w')+r_0$, so that $B_{r_0}(G',w') \subseteq
  B_R(G',u')$. Then, since the sequence $(G,u_r)$ converges to
  $(G',u')$, for all $r \geq R$ it holds that
  $B_R(G,u_r)\cong~B_R(G',u')$. Therefore, for all $r>R$ there exists
  a vertex $w_r \in B_R(G,u_r)$ such that $B_{r_0}(G,w_r) \cong
  B_{r_0}(G',w')$. Hence $(G,w_r) \in P^{(r_0)}$ for all $r>R$ (see
  Fig~\ref{fig:local_property}). Furthermore, for $r,s > R$, the balls
  $B_R(G,u_r)$ and $B_R(G,u_s)$ are disjoint, and so $w_r \neq w_s$.

  We have therefore shown that the vertices $\{w_r\}_{r > R}$ are an
  infinite set of distinct vertices such that $(G,w_r) \in P^{(r_0)}$,
  as required.

\end{proof}

\subsection{Coupling isomorphic balls}
\label{sec:isomorphics-balls}
This section includes three claims that we will use repeatedly
later. Their spirit is that everything that happens to an agent up to
time $t$ depends only on the state of the world and a ball of radius
$t$ around it.

Recall that $\cF_u(t)$, the information available to agent $u$ at time
$t$, is the algebra generated by $\psignal_u$ and $\action_w(t')$ for all
$w$ neighbors of $u$ and $t' < t$. Recall that $I_u(t)$ denotes this
exact set of actions:
\begin{align*}
  I_u(t) = \left\{ \bar{\action}_w(t):\:w \in N(u)\right\} = \left\{
    \action_w(t'):\:w \in N(u),t'<t\right\}.
\end{align*}
\begin{claim}
  \label{clm:a-from_cf}
  For all agents $u$ and times $t$, $I_u(t)$ a deterministic
  function of $\psignal(B_t(G,u))$.
\end{claim}
Recall (Definition~\ref{def:group-signal}) that $\psignal(B_t(G,u))$
are the private signals of the agents in $B_t(G,u)$, the ball of
radius $t$ around $u$ (Definition~\ref{def:balls}).
\begin{proof}
  We prove by induction on $t$. $I_u(1)$ is empty, and so the claim
  holds for $t=1$.

  Assume the claim holds up to time $t$. By definition,
  $\action_u(t+1)$ is a function of $\psignal_u$ and of $I_u(t+1)$,
  which includes $\{\action_w(t'):w\in N(u), t' \leq
  t\}$. $\action_w(t')$ is a function of $\psignal_w$ and $I_w(t')$,
  and hence by the inductive assumption it is a function of
  $\psignal(B_{t'}(G,w))$. Since $t'<t+1$ and the distance between $u$
  and $w$ is one, $W(B_{t'}(G,w)) \subseteq W(B_{t+1}(G,u))$, for all
  $w \in N(u)$ and $t' \leq t$ . Hence $I_u(t+1)$ is a function of
  $\psignal(B_{t+1}(G,u))$, the private signals in $B_{t+1}(G,u)$.
\end{proof}
The following lemma follows from Claim~\ref{clm:a-from_cf} above:
\begin{lemma}
  \label{lemma:p-iso}
  Consider two processes with identical private signal distributions
  $(\mu_0,\mu_1)$, on different graphs $G = (V,E)$ and $G' = (V',E')$.

  Let $t \geq 1$, $u \in V$ and $u' \in V'$ be such that there exists
  a rooted graph isomorphism $h:B_t(G,u) \to B_t(G',w')$.

  Let $M$ be a random variable that is measurable in $\cF_u(t)$. Then
  there exists an $M'$ that is measurable in $\cF_{u'}(t)$ such that
  the distribution of $(M,S)$ is identical to the
  distribution of $(M',S')$.
\end{lemma}
Recall that a graph isomorphism between $G=(V,E)$ and
$G'=(V',E')$ is a bijective function $h :V \to V$ such that $(u,v) \in
E$ iff $(h(u),h(v)) \in E'$.
\begin{proof}
  Couple the two processes by setting $S=S'$, and letting
  $\psignal_w=\psignal_{w'}$ when $h(w)=w'$. Note that it follows that
  $W_u=W_{u'}$. By Claim~\ref{clm:a-from_cf} we have that
  $I_u(t)=I_{u'}(t)$, when using $h$ to identify vertices in $V$ with
  vertices in $V'$.

  Since $M$ is measurable in $\cF_u(t)$, it must, by the definition of
  $\cF_u(t)$, be a function of $I_u(t)$ and $\psignal_u$. Denote then
  $M=f(I_u(t),\psignal_u)$. Since we showed that $I_u(t)=I_{u'}(t)$,
  if we let $M'=f(I_{u'}(t),\psignal_{u'})$ then the distribution of
  $(M,S)$ and $(M',S')$ will be identical.
\end{proof}
In particular, we use this lemma in the case where $M$ is an estimator
of $S$. Then this lemma implies that the probability that $M=S$ is
equal to the probability that $M'=S'$.

Recall that $p_u(t) = \P{\action_u(t) = S} = \max_{A \in
  \cF_u(t)}\P{A=S}$. Hence we can apply this lemma (\ref{lemma:p-iso})
above to $\action_u(t)$ and $\action_{u'}(t)$:
\begin{corollary}
  \label{cor:p-iso}
  If $B_t(G,u)$ and $B_t(G',u')$ are isomorphic then $p_u(t) =
  p_{u'}(t)$.
\end{corollary}

\subsection{$\delta$-independence}
\label{sec:delta-ind}
To prove that agents learn $S$ we will show that the agents must, over
the duration of this process, gain access to a large number of
measurements of $S$ that are {\em almost} independent. To formalize
the notion of almost-independence we define $\delta$-independence and
prove some easy results about it. The proofs in this subsection are
relatively straightforward.

Let $\mu$ and $\nu$ be two measures defined on the same space. We
denote the total variation distance between them by
$\dtv(\mu,\nu)$. Let $A$ and $B$ be two random variables with joint
distribution $\mu_{(A,B)}$. Then we denote by $\mu_A$ the marginal
distribution of $A$, $\mu_B$ the marginal distribution of $B$, and
$\mu_A\times\mu_B$ the product distribution of the marginal distributions.

\begin{definition}
  Let $(X_1,X_2,\dots,X_k)$ be random variables.  We refer to them as {\bf
    $\delta$-independent} if their joint distribution
  $\mu_{(X_1,\ldots,X_k)}$ has total
  variation distance of at most $\delta$ from the product of their
  marginal distributions $\mu_{X_1}\times\cdots\times\mu_{X_k}$:
  \begin{align*}
    \dtv(\mu_{(X_1,\ldots,X_k)}, \mu_{X_1}\times\cdots\times\mu_{X_k}) \leq \delta.
  \end{align*}
\end{definition}
Likewise, $(X_1,\ldots,X_l)$ are {\bf $\delta$-dependent} if the
distance between the distributions is more than $\delta$.

We remind the reader that a coupling $\nu$, between two random
variables $A_1$ and $A_2$ distributed $\nu_1$ and $\nu_2$, is a
distribution on the product of the spaces $\nu_1,\nu_2$ such that the
marginal of $A_i$ is $\nu_i$. The total variation distance between
$A_1$ and $A_2$ is equal to the minimum, over all such couplings
$\nu$, of $\nu(A_1 \neq A_2)$.

Hence to prove that $X,Y$ are
$\delta$-independent it is sufficient to show that there exists a
coupling $\nu$ between $\nu_1$, the joint distribution of $(X,Y)$ and
$\nu_2$, the products of the marginal distributions of $X$ and $Y$,
such that $\nu((X_1,Y_1) \neq (X_2,Y_2)) \leq \delta$.

Alternatively, to prove that $(A,B)$ are $\delta$-independent, one
could directly bound the total variation distance between
$\mu_{(A,B)}$ and $\mu_A\times\mu_B$ by $\delta$. This is often done
below using the fact that the total variation distance satisfies the
triangle inequality $\dtv(\mu,\nu) \leq
\dtv(\mu,\gamma)+\dtv(\gamma,\nu)$.

We state and prove some straightforward claims regarding
$\delta$-independence.

\begin{claim}
  \label{clm:delta-independent}
  Let $A$, $B$ and $C$ be random variables such that $\P{A \neq B} \leq
  \delta$ and $(B,C)$ are $\delta'$-independent. Then $(A,C)$ are
  $2\delta+\delta'$-independent.
\end{claim}
\begin{proof}
  Let $\mu_{(A,B,C)}$ be a joint distribution of $A$, $B$ and $C$ such
  that $\P{A \neq B} \leq \delta$.

  Since $\P{A \neq B} \leq \delta$, $\P{(A,C) \neq (B,C)} \leq
  \delta$, in both cases that $A,B,C$ are picked from either
  $\mu_{(A,B,C)}$ or $\mu_{(A,B)}\times\mu_C$. Hence
  \begin{align*}
    \dtv(\mu_{(A,C)},\mu_{(B,C)}) \leq \delta
  \end{align*}
  and
  \begin{align*}
    \dtv(\mu_A\times\mu_C,\mu_B\times\mu_C) \leq \delta.
  \end{align*}

  Since $(B,C)$ are $\delta'$-independent,
  \begin{align*}
    \dtv(\mu_B\times\mu_C,\mu_{(B,C)}) \leq \delta'.
  \end{align*}

  The claim follows from the triangle inequality
  \begin{align*}
    \dtv(\mu_{(A,C)},\mu_A\times\mu_C) &\leq
    \dtv(\mu_{(A,C)},\mu_{(B,C)}) + \dtv(\mu_{(B,C)},\mu_B\times\mu_C)
    + \dtv(\mu_B\times\mu_C,\mu_A\times\mu_C)\\
    &\leq 2\delta+\delta'.
  \end{align*}

\end{proof}

\begin{claim}
  \label{clm:function-independent}
  Let $(X,Y)$ be $\delta$-independent, and let $Z = f(Y,B)$ for some
  function $f$ and $B$ that is independent of both $X$ and $Y$. Then
  $(X,Z)$ are also $\delta$-independent.
\end{claim}
\begin{proof}
  Let $\mu_{(X,Y)}$ be a joint distribution of $X$ and $Y$ satisfying
  the conditions of the claim.  Then since $(X,Y)$ are
  $\delta$-independent,
  \begin{align*}
    \dtv(\mu_{(X,Y)},\mu_X\times\mu_Y) \leq \delta.
  \end{align*}
  Since $B$ is independent of both $X$ and $Y$,
%  $\mu_{(X,Y,B)}=\mu_{(X,Y)}\times\mu_B$ and so
  \begin{align*}
    \dtv(\mu_{(X,Y)}\times\mu_B,\mu_X\times\mu_Y\times\mu_B) \leq \delta
  \end{align*}
  and $(X,Y,B)$ are $\delta$-independent. Therefore there exists a
  coupling between $(X_1,Y_1,B_1) \sim \mu_{(X,Y)}\times\mu_B$ and
  $(X_2,Y_2,B_2) \sim \mu_X\times\mu_Y\times\mu_B$ such that
  $\P{(X_1,Y_1,B_1) \neq (X_2,Y_2,B_2)} \leq \delta$. Then
  \begin{align*}
    \P{(X_1,f(Y_1,B_1)) \neq (X_2,f(Y_2,B_2))} \leq \delta
  \end{align*}
  and the proof follows.
\end{proof}

\begin{claim}
  \label{clm:delta-ind-additive}
  Let $A=(A_1,\ldots,A_k)$, and $X$ be random variables. Let
  $(A_1,\ldots,A_k)$ be $\delta_1$-independent and let $(A,X)$ be
  $\delta_2$-independent. Then $(A_1,\ldots,A_k,X)$ are
  $(\delta_1+\delta_2)$-independent.
\end{claim}
\begin{proof}
  Let $\mu_{(A_1,\ldots,A_k,X)}$ be the joint distribution of
  $A=(A_1,\ldots,A_k)$ and $X$. Then since $(A_1,\ldots,A_k)$ are
  $\delta_1$-independent,
  \begin{align*}
    \dtv(\mu_A,\mu_{A_1}\times\cdots\times\mu_{A_k})
    \leq \delta_1.
  \end{align*}
  Hence
  \begin{align*}
    \dtv(\mu_A\times\mu_X,\mu_{A_1}\times\cdots\times\mu_{A_k}\times\mu_X)
    \leq \delta_1.
  \end{align*}

  Since $(A,X)$ are $\delta_2$-independent,
  \begin{align*}
    \dtv(\mu_{(A,X)},\mu_A\times\mu_X) \leq \delta_2.
  \end{align*}

  The claim then follows from the triangle inequality
  \begin{align*}
    \dtv(\mu_{(A,X)},\mu_{A_1}\times\cdots\times\mu_{A_k}\times\mu_X)
    \leq
    \dtv(\mu_{(A,X)},\mu_A\times\mu_X)
    +
    \dtv(\mu_A\times\mu_X,\mu_{A_1}\times\cdots\times\mu_{A_k}\times\mu_X).
  \end{align*}
\end{proof}

\begin{lemma}
  \label{cor:majority}
  For every $1/2 < p < 1$ there exist $\delta = \delta(p) >0$ and
  $\eta = \eta(p) > 0$  such that if $S$ and
  $(X_1,X_2,X_3)$ are binary random variables with $\P{S=1}=1/2$, $1/2
  < p-\eta \leq \P{X_i=S} < 1$, and $(X_1,X_2,X_3)$ are
  $\delta$-independent conditioned on $S$ then $\P{a(X_1,X_2,X_3) = S}
  > p$, where $a$ is the MAP estimator of $S$ given $(X_1,X_2,X_3)$.
\end{lemma}
In other words, one's odds of guessing $S$ using three conditionally
almost-independent bits are greater than using a single bit.
\begin{proof}
  We apply Lemma~\ref{clm:majority} below to three conditionally
  independent bits which are each equal to $S$ w.p.\ at least
  $p-\eta$. Then
  \begin{align*}
    \P{a(X_1,X_2,X_3) = S} \geq p-\eta+\eps_{p-\eta}
  \end{align*}
  where $\eps_q=\frac{1}{100}(2q-1)(3q^2-2q^3-q)$.

  Since $\eps_q$ is continuous in $q$ and positive for $1/2<q<1$, it
  follows that for $\eta$ small enough $p-\eta+\eps_{p-\eta} >
  p$. Now, take $\delta < \eps_{p-\eta}-\eta$. Then, since we can
  couple $\delta$-independent bits to independent bits so that they
  differ with probability at most $\delta$, the claim follows.
\end{proof}

\begin{lemma}
  \label{clm:majority}
  Let $S$ and $(X_1,X_2,X_3)$ be binary random variables such that
  $\P{S=1}=1/2$. Let $1/2 < p \leq \P{X_i=S} < 1$. Let
  $a(X_1,X_2,X_3)$ be the MAP estimator of $S$ given
  $(X_1,X_2,X_3)$. Then there exists an $\eps_p>0$ that depends only on
  $p$ such that if $(X_1,X_2,X_3)$ are independent conditioned on $S$
  then $\P{a(X_1,X_2,X_3) = S} \geq p + \eps_p$.

  In particular the statement holds with
  \begin{align*}
  \eps_p=\frac{1}{100}(2p-1)(3p^2-2p^3-p).
  \end{align*}
\end{lemma}
\begin{proof}
  Denote $X=(X_1,X_2,X_3)$.

  Assume first that $\P{X_i=S} = p$ for all $i$.  Let $\delta_1,
  \delta_2, \delta_3$ be such that $p+\delta_i = \CondP{X_i=1}{S=1}$
  and $p-\delta_i=\CondP{X_i=0}{S=0}$.

  To show that $\P{a(X) = S} \geq p + \eps_p$ it is enough
  to show that $\P{b(X) = S} \geq p + \eps_p$ for some
  estimator $b$, by the definition of a MAP estimator. We separate
  into three cases.
  \begin{enumerate}
  \item If $\delta_1 = \delta_2 = \delta_3 = 0$ then the events
    $X_i=S$ are independent and the majority of the $X_i$'s is equal
    to $S$ with probability $p' = p^3+3p^2(1-p)$, which is greater than $p$
    for $\half < p < 1$. Denote $\eta_p = p' - p$. Then
    $\P{a(X) = S} \geq p + \eta_p$.

  \item Otherwise if $|\delta_i| \leq \eta_p / 6$ for all $i$ then we
    can couple $X$ to three bits $Y=(Y_1,Y_2,Y_3)$ which satisfy the
    conditions of case 1 above, and so that $\P{X \neq Y} \leq
    \eta_p/2$. Then $\P{a(X) = S} \geq p + \eta_p/2$.

  \item Otherwise we claim that there exist $i$ and $j$ such that
    $|\delta_i+\delta_j| > \eta_p/12$.

    Indeed assume w.l.o.g.\ that $\delta_1 \geq \eta_p/6$. Then if it
    doesn't hold that $\delta_1+\delta_2 \geq \eta_p/12$ and it
    doesn't hold that $\delta_1+\delta_3 \geq \eta_p/12$ then
    $\delta_2 \leq -\eta_p/12$ and $\delta_3 \leq -\eta_p/12$ and
    therefore $\delta_2+\delta_3 \leq -\eta_p/12$.

    Now that this claim is proved, assume w.l.o.g.\ that
    $\delta_1+\delta_2 \geq \eta_p/12$. Recall that $X_i \in \{0,1\}$,
    and so the product $X_1X_2$ is also an element of $\{0,1\}$. Then
  \begin{align*}
    \P{X_1X_2 = S} &= \half\CondP{X_1X_2 = 1}{S=1}+\half\CondP{X_1X_2 =
      0}{S=0}\\
    &= \half\left((p+\delta_1)(p+\delta_2)+(p-\delta_1)(p-\delta_2)+(p-\delta_1)(1-p+\delta_2)+(1-p+\delta_1)(p-\delta_2)\right)\\
    &= p+\half(2p-1)(\delta_1+\delta_2)\\
    &\geq p + (2p-1)\eta_p/12,
  \end{align*}
  and so $\P{a(X) = S} \geq p + (2p-1)\eta_p/12$.
  \end{enumerate}

  Finally, we need to consider the case that $\P{X_i=S} = p_i > p$ for
  some $i$. We again consider two cases. Denote $\eps_p =
  (2p-1)\eta_p/100$. If there exists an $i$ such that $p_i > \eps_p$
  then this bit is by itself an estimator that equals $S$ with
  probability at least $p+\eps_p$, and therefore the MAP
  estimator equals $S$ with probability at least $p+\eps_p$.

  Otherwise $p \leq p_i \leq p_i+\eps_p$ for all $i$.  We will
  construct a coupling between the distributions of $X=(X_1,X_2,X_3)$
  and $Y=(Y_1,Y_2,Y_3)$ such that the $Y_i$'s are conditionally
  independent given $S$ and $\P{Y_i=S} = p$ for all $i$, and
  furthermore $\P{Y \neq X} \leq 3\eps_p$. By what we've proved so far
  the MAP estimator of $S$ given $Y$ equals $S$ with probability at
  least $p+(2p-1)\eta_p/12 \geq p+8\eps_p$. Hence by the coupling, the
  same estimator applied to $X$ is equal to $S$ with probability at
  least $p+8\eps_p-3\eps_p>p+\eps_p$.

  To couple $X$ and $Y$ let $Z_i$ be a real i.i.d.\ random variables
  uniform on $[0,1]$. When $S=1$ let $X_i=Y_i=S$ if $Z_i >
  p_i+\delta_i$, let $X_i=S$ and $Y_i=1-S$ if $Z_i \in
  [p+\delta_i,p_i+\delta_i]$, and otherwise $X_i=Y_i=1-S$. The
  construction for $S=0$ is similar. It is clear that $X$ and $Y$ have the
  required distribution, and that furthermore $\P{X_i \neq Y_i} =
  p_i-p \leq \eps_p$. Hence $\P{X \neq Y} \leq 3\eps_p$, as needed.

\end{proof}

\subsection{Asymptotic learning}
\label{sec:learning}

In this section we prove Theorem~\ref{thm:bounded-learning}.
\begin{theorem*}[\ref{thm:bounded-learning}]
  Let $\mu_0,\mu_1$ be such that for every connected, undirected graph
  $G$ there exists a random variable $\limac$ such that almost surely
  $\limac_u=\limac$ for all $u \in V$.  Then there exists a sequence
  $q(n)=q(n, \mu_0, \mu_1)$ such that $q(n) \to 1$ as $n \to \infty$,
  and $\P{\limac = \{S\}} \geq q(n)$, for any choice of undirected,
  connected graph $G$ with $n$ agents.
\end{theorem*}

To prove this theorem we will need a number of intermediate results,
which are given over the next few subsections.

\subsubsection{Estimating the limiting optimal action set $\limac$}
We would like to show that although the agents have a common optimal
action set $\limac$ only at the limit $t \to \infty$, they can estimate
this set well at a large enough time $t$.

The action $\action_u(t)$ is agent $u$'s MAP estimator of $S$ at time
$t$ (see Remark~\ref{rem:map}). We likewise define $\estL_u(t)$ to be
agent $u$'s MAP estimator of $\limac$, at time $t$:
\begin{align}
  \label{eq:l-i-t}
  \estL_u(t) = \argmax_{\estL \in \{\{0\},\{1\},\{0,1\}\}}\CondP{\limac = \estL}{\cF_u(t)}.
\end{align}
We show that the sequence of random variables $\estL_u(t)$ converges
to $\limac$ for every $u$, or that alternatively $\estL_u(t)=\limac$ for each
agent $u$ and $t$ large enough:
\begin{lemma}
  \label{lemma:L-estimate}
  $\P{\lim_{t \to \infty}\estL_u(t) = \limac} = 1$ for all $u \in V$.
\end{lemma}
This lemma (\ref{lemma:L-estimate}) follows by direct application of
the more general Lemma~\ref{lemma:estimation-converges} which we prove
below.  Note that a consequence is that $\lim_{t \to \infty}
\P{\estL_u(t) = \limac} = 1$.

\begin{lemma}
  \label{lemma:estimation-converges}
  Let $\cK_1 \subseteq \cK_2, \ldots$ be a filtration of
  $\sigma$-algebras, and let $\cK_\infty = \cup_t\cK_t$. Let $K$ be a
  random variable that takes a finite number of values and is
  measurable in $\cK_\infty$. Let $M(t)=\argmax_k\CondP{K =
    k}{\cK(t)}$ be the MAP estimator of $K$ given $\cK_t$. Then
  \begin{align*}
    \P{\lim_{t \to \infty} M(t)=K} = 1.
  \end{align*}
\end{lemma}
\begin{proof}
  For each $k$ in the support of $K$, $\CondP{K=k}{\cK_t}$ is a
  bounded martingale which converges almost surely to
  $\CondP{K=k}{\cK_\infty}$, which is equal to $\ind{K=k}$, since $K$
  is measurable in $G_\infty$.  Therefore
  $M(t)=\argmax_k\CondP{K=k}{\cK_t}$ converges almost surely to
  $\argmax_k\CondP{K=k}{\cK_\infty}=K$.
\end{proof}

We would like at this point to provide the reader with some more
intuition on $\action_u(t)$, $\estL_u(t)$ and the difference between
them. Assuming that $\limac = \{1\}$ then by definition, from some
time $t_0$ on, $\action_u(t)=1$, and from
Lemma~\ref{lemma:L-estimate}, $\estL_u(t)=\{1\}$. The same applies
when $\limac = \{0\}$. However, when $\limac = \{0,1\}$ then
$\action_u(t)$ may take both values 0 and 1 infinitely often, but
$\estL_u(t)$ will eventually equal $\{0,1\}$. That is, agent $u$ will
realize at some point that, although it thinks at the moment that 1 is
preferable to 0 (for example), it is in fact the most likely outcome
that its belief will converge to $1/2$. In this case, although it is
not optimal, a {\em uniformly random} guess of which is the best
action may not be so bad. Our next definition is based on this
observation.

Based on $\estL_u(t)$, we define a second ``action'' $C_u(t)$.
\begin{definition}
  Let $C_u(t)$ be picked uniformly from $\estL_u(t)$: if $\estL_u(t) =
  \{1\}$ then $C_u(t) = 1$, if $\estL_u(t) = \{0\}$ then $C_u(t) = 0$,
  and if $\estL_u(t) = \{0,1\}$ then $C_u(t)$ is picked independently
  from the uniform distribution over $\{0,1\}$.
\end{definition}

Note that we here extend our probability space by including in
$I_u(t)$ (the observations of agent $u$ up to time $t$) an extra
uniform bit that is independent of all else and $S$ in
particular. Hence this does not increase $u$'s ability to estimate
$S$, and if we can show that in this setting $u$ learns $S$ then $u$
can also learn $S$ without this bit.  In fact, we show that
asymptotically it is as good an estimate for $S$ as the best estimate
$\action_u(t)$:
\begin{claim}
  \label{claim:b-a-equiv}
  $\lim_{t \to \infty} \P{C_u(t)=S} = \lim_{t \to \infty} \P{\action_u(t)=S}
  = p$ for all $u$.
\end{claim}
\begin{proof}
  We prove the claim by showing that it holds both when conditioning
  on the event $\limac = \{0,1\}$ and when conditioning on its
  complement.

  When $\limac \neq \{0,1\}$ then for $t$ large enough
  $\limac=\{\action_u(t)\}$. Since (by Lemma~\ref{lemma:L-estimate})
  $\lim \estL_u(t) = \limac$ with probability 1, in this case
  $C_u(t)=\action_u(t)$ for $t$ large enough, and
  \begin{align*}
    \lim_{t \to \infty}\CondP{C_u(t)=S}{\limac \neq \{0,1\}}\;=\;
    \CondP{\limac=\{S\}}{\limac \neq \{0,1\}} \;=\; \lim_{t \to
      \infty}\CondP{\action_u(t)=S}{\limac \neq \{0,1\}}.
  \end{align*}

  When $\limac = \{0,1\}$ then $\lim \belief_u(t) = \lim
  \CondP{\action_u(t)=S}{\cF_u(t)} = 1/2$ and so $\lim \P{\action_u(t)=S} =
  1/2$. This is again also true for $C_u(t)$, since in this case it is
  picked at random for $t$ large enough, and so
  \begin{align*}
    \lim_{t \to \infty}\CondP{C_u(t)=S}{\limac = \{0,1\}}\;=\; \frac{1}{2} \;=\;
    \lim_{t \to \infty}\CondP{\action_u(t)=S}{\limac = \{0,1\}}.
  \end{align*}
\end{proof}

\subsubsection{The probability of getting it right}
Recall Definition~\ref{def:p-u}: $p_u(t) = \P{\action_u(t)=S}$ and
$p_u = \lim_{t \to \infty}p_u(t)$ (i.e., $p_u(t)$ is the probability
that agent $u$ takes the right action at time $t$). We prove here a
few easy related claims that will later be useful to us.
\begin{claim}
  \label{clm:pMonotone}
  $p_u(t+1) \geq p_u(t)$.
\end{claim}
\begin{proof}
  Condition on $\cF_u(t+1)$, the information available to agent $u$ at
  time $t+1$. Hence the probability that $\action_u(t+1) = S$ is at
  least as high as the probability $\action_u(t) = S$, since
  \begin{align*}
    \action_u(t+1)=\argmax_s\CondP{S=s}{\cF(t+1)}
  \end{align*}
  and $\action_u(t)$ is measurable in $\cF(t+1)$.  The claim is proved
  by integrating over all possible values of $\cF_u(t+1)$.
\end{proof}
Since $p_u(t)$ is bounded by one, Claim~\ref{clm:pMonotone} means that
the limit $p_u$ exists. We show that this value is the same for all
vertices.
\begin{claim}
  \label{clm:pG}
  There exists a $p \in [0,1]$ such that $p_u=p$ for all $u$.
\end{claim}
\begin{proof}
  Let $u$ and $w$ be neighbors. As in the proof above, we can argue
  that $\CondP{\action_u(t+1) = S} {\cF_u(t+1)} \geq
  \CondP{\action_w(t) = S} {\cF_u(t+1)}$, since $\action_w(t)$ is
  measurable in $\cF_u(t+1)$. Hence the same holds unconditioned, and
  so we have that $p_u \geq p_w$, by taking the limit $t \to
  \infty$. Since the same argument can be used with the roles of $u$
  and $w$ reversed, we have that $p_u=p_w$, and the claim follows from
  the connectedness of the graph, by induction.
\end{proof}

We make the following definition in the spirit of these claims:
\begin{definition}
  $p = \lim_{t \to \infty} \P{\action_u(t)=S}$.
\end{definition}
In the context of a specific social network graph $G$ we may denote
this quantity as $p(G)$.

For time $t=1$ the next standard claim follows from the fact that the
agents' signals are informative.
\begin{claim}
  \label{clm:pGreaterThanHalf}
  $p_u(t)>1/2$ for all $u$ and $t$.
\end{claim}
\begin{proof}
  Note that
  \begin{align*}
    \CondP{\action_u(1)=S}{\psignal_u} =
    \max\{\belief_u(1),1-\belief_u(1)\} =
    \max\{\CondP{S=0}{\psignal_u}, \CondP{S=1}{\psignal_u}\}.
  \end{align*}
  Recall that $p_u(1) = \P{\action_u(1) = S}$. Hence
  \begin{align*}
    p_u(1) &= \E{\CondP{\action_u(1)=S}{\psignal_u}}\\
    &= \E{\max\{\CondP{S=0}{\psignal_u}, \CondP{S=1}{\psignal_u}\}}\\
  \end{align*}
  Since $\max\{a,b\}=\half(a+b)+\half|a-b|$, and since
  $\CondP{S=0}{\psignal_u}+ \CondP{S=1}{\psignal_u}=1$, it follows that
  \begin{align*}
    p_u(1) &=
    \half+\half\E{|\CondP{S=0}{\psignal_u}-\CondP{S=1}{\psignal_u}|}\\
    &= \half+\half D_{TV}(\mu_0,\mu_1),
  \end{align*}
  where the last equality follows by Bayes' rule.  Since $\mu_0 \neq
  \mu_1$, the total variation distance $D_{TV}(\mu_0,\mu_1)>0$ and
  $p_u(1) > \half$. For $t>1$ the claim follows from
  Claim~\ref{clm:pMonotone} above.
\end{proof}

Recall that $|N(u)|$ is the out-degree of $u$, or the number of
neighbors that $u$ observes. The next lemma states that an agent with
many neighbors will have a good estimate of $S$ already at the second
round, after observing the first action of its neighbors. This lemma
is adapted from Mossel and Tamuz~\cite{MosselTamuz10B:arxiv}, and
provided here for completeness.
\begin{lemma}
  \label{thm:large-out-deg}
  There exist constants $C_1=C_1(\mu_0,\mu_1)$ and
  $C_2=C_2(\mu_0,\mu_1)$ such that for any agent $u$ it holds that
  \begin{align*}
    p_u(2) \geq 1-C_1 e^{-C_2 \cdot N(u)}.
  \end{align*}
\end{lemma}

\begin{proof}
  Conditioned on $S$, private signals are independent and identically
  distributed. Since $\action_w(1)$ is a deterministic function of
  $\psignal_w$, the initial actions $\action_w(1)$ are also
  identically distributed, conditioned on $S$. Hence there exists a
  $q$ such that $p_w(1) = \P{\action_w(t)=S} = q$ for all agents
  $w$. By Lemma~\ref{clm:pGreaterThanHalf} above, $q>1/2$. Therefore
  \begin{align*}
    \CondP{\action_w(1)=1}{S=1} \neq \CondP{\action_w(1)=1}{S=0},
  \end{align*}
  and the distribution of $\action_w(1)$ is different when conditioned
  on $S=0$ or $S=1$.

  Fix an agent $u$, and let $n = |N(u)|$ be the out-degree of $u$, or
  the number of neighbors that it observes. Let
  $\{w_1,\ldots,w_{|N(u)|}\}$ be the set of $u$'s neighbors. Recall
  that $A_u(2)$ is the MAP estimator of $S$ given
  $(A_{w_1}(1),\ldots,A_{w_n}(1))$, and given $u$'s private signal.

  By standard asymptotic statistics of hypothesis testing
  (cf.~\cite{dasgupta2008asymptotic}), testing an hypothesis (in our
  case, say, $S=1$ vs.\ $S=0$) given $n$ informative, conditionally
  i.i.d.\ signals, succeeds except with probability that is
  exponentially low in $n$. It follows that $\P{A_u(2) \neq S}$ is
  exponentially small in $n$, so that there exist $C_1$ and $C_2$ such
  that
  \begin{align*}
    p_u(2) = \P{\action_u(2) = S} \geq 1-C_1 e^{-C_2 \cdot N(u)}.
  \end{align*}
\end{proof}

The following claim is a direct consequence of the previous lemmas of
this section.
\begin{claim}
  \label{thm:large-out-deg-graph}
  Let $d(G) = \sup_u\{N(u)\}$ be the out-degree of the graph $G$; note
  that for infinite graphs it may be that $d=\infty$. Then there exist
  constants $C_1=C_1(\mu_0,\mu_1)$ and $C_2=C_2(\mu_0,\mu_1)$ such
  that
  \begin{align*}
    p(G) \geq 1-C_1 e^{-C_2 \cdot d(G)}
  \end{align*}
  for all agents $u$.
\end{claim}
\begin{proof}
  Let $u$ be an arbitrary vertex in $G$. Then by
  Lemma~\ref{thm:large-out-deg} it holds that
  \begin{align*}
    p_u(2) \geq 1-C_1 e^{-C_2 \cdot N(u)},
  \end{align*}
  for some constants $C_1$ and $C_2$. By Lemma~\ref{clm:pMonotone} we
  have that $p_u(t+1) \geq p_u(t)$, and therefore
  \begin{align*}
    p_u  = \lim_{n \to \infty}p_u(t) \geq 1-C_1 e^{-C_2 \cdot N(u)}.
  \end{align*}
  Finally, $p(G) = p_u$ by Lemma~\ref{clm:pG}, and so
  \begin{align*}
    p_u  \geq 1-C_1 e^{-C_2 \cdot N(u)}.
  \end{align*}
  Since this holds for an arbitrary vertex $u$, the claim follows.
\end{proof}

\subsubsection{Local limits and pessimal graphs}
We now turn to apply local limits to our process. We consider here and
henceforth the same model of Definitions~\ref{def:state-signals}
and~\ref{def:revealed-actions}, as applied, with the same private
signals, to different graphs. We write $p(G)$ for the value of $p$ on
the process on $G$, $\limac(G)$ for the value of $\limac$ on $G$, etc.

\begin{lemma}
  \label{lemma:limit-leq-p}
  Let $(G,u) = \lim_{r \to \infty}(G_r,u_r)$. Then $p(G) \leq \lim
  \inf_r p(G_r)$.
\end{lemma}
\begin{proof}
  Since $B_r(G_r,u_r) \cong B_r(G,u)$, by Lemma~\ref{cor:p-iso} we
  have that $p_u(r) = p_{u_r}(r)$. By Claim~\ref{clm:pMonotone}
  $p_{u_r}(r) \leq p(G_r)$, and therefore $p_u(r) \leq p(G_r)$. The
  claim follows by taking the $\liminf$ of both sides.
\end{proof}
A particularly interesting case in the one the different $G_r$'s are
all the same graph:
\begin{corollary}
  \label{cor:limit-leq-p}
  Let $G$ be a (perhaps infinite) graph, and let $\{u_r\}$ be a
  sequence of vertices. Then if the local limit $(H,u) = \lim_{r \to
    \infty}(G,u_r)$ exists then $p(H) \leq p(G)$.
\end{corollary}

Recall that $\InfGraphs_d$ denotes the set of infinite, connected,
undirected graphs of degree at most $d$. Let
\begin{align*}
  \InfGraphs = \bigcup_d \InfGraphs_d.
\end{align*}
\begin{definition}
  Let
  \begin{align*}
    p^* = p^*(\mu_0,\mu_1) = \inf_{G \in \InfGraphs}p(G)
  \end{align*}
  be the probability of learning in the pessimal graph.
\end{definition}
Note that by Claim~\ref{clm:pGreaterThanHalf} we have that $p^*>1/2$.
We show that this infimum is in fact attained by some graph:
\begin{lemma}
  \label{lemma:h-exists}
  There exists a graph $H \in \InfGraphs$ such that $p(H)=p^*$.
\end{lemma}
\begin{proof}
  Let $\{G_r = (V_r,E_r)\}_{r=1}^{\infty}$ be a series of graphs in
  $\InfGraphs$ such that $\lim_{r \to \infty} p(G_r) = p^*$. Note that
  $\{G_r\}$ must all be in $\InfGraphs_d$ for some $d$ (i.e., have
  uniformly bounded degrees), since otherwise the sequence $p(G_r)$
  would have values arbitrarily close to $1$ and its limit could not
  be $p^*$ (unless indeed $p^* = 1$, in which case our main
  Theorem~\ref{thm:bounded-learning} is proved). This follows from
  Lemma~\ref{thm:large-out-deg}.

  We now arbitrarily mark a vertex $u_r$ in each graph, so that $u_r
  \in V_r$, and let $(H,u)$ be the limit of some
  subsequence of $\{G_r,u_r\}_{r=1}^\infty$. Since $\InfGraphs_d$ is
  compact (Lemma~\ref{lemma:inf-graphs-closed}), $(H,u)$ is guaranteed
  to exist, and $H \in \InfGraphs_d$.

  By Lemma~\ref{lemma:limit-leq-p} we have that $p(H) \leq \liminf_r
  p(G_r) = p^*$. But since $H \in \InfGraphs$, $p(H)$ cannot be less
  than $p^*$, and the claim is proved.
\end{proof}

\subsubsection{Independent bits}
We now show that on infinite graphs, the private signals in the
neighborhood of agents that are ``far enough away'' are (conditioned
on $S$) almost independent of $\limac$ (the final consensus estimate
of $S$).

\begin{lemma}
  \label{lemma:r-independent}
  Let $G$ be an infinite graph. Fix a vertex $u_0$ in $G$.  Then for
  every $\delta>0$ there exists an $r_\delta$ such that for every
  $r\geq r_\delta$ and every vertex $u$ with $d(u_0,u)>2r$ it holds
  that $\psignal(B_r(G,u))$, the private signals in $B_r(G,u)$, are
  $\delta$-independent of $\limac$, conditioned on $S$.
\end{lemma}
Here we denote graph distance by $d(\cdot,\cdot)$.
\begin{proof}
  Fix $u_0$, and let $u$ be such that $d(u_0,u) > 2r$. Then
  $B_r(G,u_0)$ and $B_r(G,u)$ are disjoint, and hence independent
  conditioned on $S$. Hence $\estL_{u_0}(r)$ is independent of
  $\psignal(B_r(G,u))$, conditioned on $S$.

  Lemma~\ref{lemma:L-estimate} states that $\P{\lim_{r \to
      \infty}\estL_{u_0}(r) = \limac} = 1$, and so there exists an
  $r_\delta$ such that for every $r \geq r_\delta$ it holds that
  $\P{\estL_{u_0}(r) = \limac} > 1-\half\delta$.

  Recall Claim~\ref{clm:delta-independent}: for any $A,B,C$, if
  $\P{A=B} = 1-\half\delta$ and $B$ is independent of $C$, then
  $(A,C)$ are $\delta$-independent.

  Applying Claim~\ref{clm:delta-independent} to $\limac$,
  $\estL_{u_0}(r)$ and $\psignal(B_r(G,u))$ we get that for any $r$
  greater than $r_\delta$ it holds that $\psignal(B_r(G,u))$ is
  $\delta$-independent of $\limac$, conditioned on $S$.
\end{proof}

We will now show, in the lemmas below, that in infinite graphs each
agent has access to any number of ``good estimators'':
$\delta$-independent measurements of $S$ that are each almost as
likely to equal $S$ as $p^*$, the minimal probability of estimating
$S$ on any infinite graph.
\begin{definition}
  We say that agent $u \in G$ has {\bf $k$ $(\delta,\eps)$-good
    estimators} if there exists a time $t$ and estimators
  $M_1,\ldots,M_k$ such that $(M_1,\ldots,M_k) \in \cF_u(t)$ and
  \begin{enumerate}
  \item $\P{M_i = S} > p^* - \eps$ for $1 \leq i \leq k$.
  \item $(M_1, \ldots, M_k)$ are $\delta$-independent, conditioned on
    $S$.
  \end{enumerate}
\end{definition}

\begin{claim}
  \label{clm:good-est-local}
  Let $P$ denote the property of having $k$ $(\delta,\eps)$-good
  estimators. Then $P$ is a {\em local property}
  (Definition~\ref{def:local-property}) of the rooted graph
  $(G,u)$. Furthermore, if $u \in G$ has $k$ $(\delta,\eps)$-good
  estimators measurable in $\cF_u(t)$ then $(G,u) \in P^{(t)}$, i.e.,
  $(G,u)$ has property $P$ with radius $t$.
\end{claim}
\begin{proof}
  If $(G,u) \in P$ then by definition there exists a time $t$ such
  that $(M_1,\ldots,M_k) \in \cF_u(t)$. Hence by
  Lemma~\ref{lemma:p-iso}, if $B_t(G,u) \cong B_t(G',u')$ then $u' \in
  G'$ also has $k$ $(\delta,\eps)$-good estimators $(M'_1,\ldots,M'_k)
  \in \cF_{u'}(t)$ and $(G',u') \in P$. In particular, $(G,u) \in
  P^{(t)}$, i.e., $(G,u)$ has property $P$ with radius $t$.
\end{proof}

We are now ready to prove the main lemma of this subsection:
\begin{lemma}
  \label{thm:independent-bits}
  For every $d \geq 2$, $G \in \InfGraphs_d$, $\eps, \delta > 0$ and
  $k \geq 0$ there exists a vertex $u$, such that $u$ has $k$
  $(\delta,\eps)$-good estimators.
\end{lemma}
Informally, this lemma states that if $G$ is an infinite graph with
bounded degrees, then there exists an agent that eventually has $k$
almost-independent estimates of $S$ with quality close to $p^*$, the
minimal probability of learning.

\begin{proof}
  In this proof we use the term ``independent'' to mean ``independent
  conditioned on $S$''.

  We choose an arbitrary $d$ and prove by induction on $k$. The basis
  $k = 0$ is trivial.  Assume the claim holds for $k$, any $G \in
  \InfGraphs_d$ and all $\eps,\delta>0$. We shall show that it holds
  for $k+1$, any $G \in \InfGraphs_d$ and any $\delta,\eps>0$.

  By the inductive hypothesis for every $G \in \InfGraphs_d$ there
  exists a vertex in $G$ that has $k$ $(\delta/100,\eps)$-good
  estimators $(M_1, \ldots, M_k)$.

  Now, having $k$ $(\delta/100,\eps)$-good estimators is a {\em local
    property} (Claim~\ref{clm:good-est-local}).  We now therefore
  apply Lemma~\ref{lemma:local_property}: since every graph $G \in
  \InfGraphs_d$ has a vertex with $k$ $(\delta/100,\eps)$-good
  estimators, any graph $G \in \InfGraphs_d$ has a time $t_k$ for
  which infinitely many distinct vertices $\{w_r\}$ have $k$
  $(\delta/100,\eps)$-good estimators measurable at time $t_k$.

  In particular, if we fix an arbitrary $u_0 \in G$ then for every $r$
  there exists a vertex $w \in G$ that has $k$
  $(\delta/100,\eps)$-good estimators and whose distance $d(u_0,w)$
  from $u_0$ is larger than $r$.

  We shall prove the lemma by showing that for a vertex $w$ that is
  far enough from $u_0$ which has $(\delta/100,\eps)$-good estimators
  $(M_1,\ldots,M_k)$, it holds that for a time $t_{k+1}$ large enough
  $(M_1,\ldots,M_k,C_w(t_{k+1}))$ are $(\delta,\eps)$-good estimators.

  By Lemma~\ref{lemma:r-independent} there exists an $r_\delta$ such
  that if $r>r_\delta$ and $d(u_0,w) > 2r$ then $\psignal(B_r(G,w))$
  is $\delta/100$-independent of $\limac$.  Let $r^* =
  \max\{r_\delta,t_k\}$, where $t_k$ is such that there are infinitely
  many vertices in $G$ with $k$ good estimators measurable at time
  $t_k$.

  Let $w$ be a vertex with $k$ $(\delta/100,\eps)$-good estimators
  $(M_1,\ldots,M_k)$ at time $t_k$, such that $d(u_0,w) > 2r^*$.
  Denote
  \begin{align*}
    \bar{M}=(M_1, \ldots, M_k).
  \end{align*}
  Since $d(u_0,w) > 2r_\delta$, $\psignal(B_{r^*}(G,w))$ is
  $\delta/100$-independent of $\limac$, and since $B_{t_k}(G,w)
  \subseteq B_{r^*}(G,w)$, $\psignal(B_{t_k}(G,w))$ is
  $\delta/100$-independent of $\limac$. Finally, since $\bar{M} \in
  \cF_w(t_k)$, $\bar{M}$ is a function of $\psignal(B_{t_k}(G,w))$,
  and so by Claim~\ref{clm:function-independent} we have that
  $\bar{M}$ is also $\delta/100$-independent of $\limac$.

  For $t_{k+1}$ large enough it holds that
  \begin{itemize}
  \item  $\estL_w(t_{k+1})$ is equal to $L$ with probability at least
    $1-\delta/100$, since
    \begin{align*}
      \lim_{t \to \infty} \P{\estL_w(t)=\limac} = 1,
    \end{align*}
     by Claim~\ref{lemma:L-estimate}.
  \item Additionally, $\P{C_w(t_{k+1})=S} > p^*-\eps$, since
    \begin{align*}
      \lim_{t \to \infty}\P{C_w(t)=S} = p \geq p^*,
    \end{align*}
    by Claim~\ref{claim:b-a-equiv}.
  \end{itemize}

  We have then that $(\bar{M},\limac)$ are $\delta/100$-independent and
  $\P{\estL_w(t_{k+1}) \neq \limac} \leq \delta/100$.
  Claim~\ref{clm:delta-independent} states that if $(A,B)$ are
  $\delta$-independent $\P{B \neq C} \leq \delta'$ then $(A,C)$ are
  $\delta+2\delta'$-independent. Applying this here we get that
  $(\bar{M},\estL_w(t_{k+1}))$ are $\delta/25$-independent.

  It follows by application of Claim~\ref{clm:delta-ind-additive} that
  $(M_1,\ldots,M_k,\estL_w(t_{k+1}))$ are $\delta$-independent.  Since
  $C_w(t_{k+1})$ is a function of $\estL_w(t_{k+1})$ and an
  independent bit, it follows by another application of
  Claim~\ref{clm:function-independent} that $(M_1, \ldots, M_k,
  C_w(t_{k+1}))$ are also $\delta$-independent.

  Finally, since $\P{C_w(t_{k+1})=S} > p^*-\eps$, $w$ has the $k+1$
  $(\delta,\eps)$-good estimators $(M_1,\ldots,C_w(t_{k+1}))$ and the
  proof is concluded.

\end{proof}
\subsubsection{Asymptotic learning}
As a tool in the analysis of finite graphs, we would like to prove
that in infinite graphs the agents learn the correct state of the
world almost surely.

\begin{theorem}
  \label{thm:bounded}
  Let $G=(V,E)$ be an infinite, connected undirected graph with
  bounded degrees (i.e., $G$ is a general graph in $\InfGraphs$). Then
  $p(G)=1$.
\end{theorem}
Note that an alternative phrasing of this theorem is that $p^*=1$.
\begin{proof}
  Assume the contrary, i.e. $p^*<1$. Let $H$ be an infinite, connected
  graph with bounded degrees such that $p(H) = p^*$, such as we've
  shown exists in Lemma~\ref{lemma:h-exists}.

  By Lemma~\ref{thm:independent-bits} there exists for arbitrarily
  small $\eps,\delta >0 $ a vertex $w \in H$ that has access at some
  time $T$ to three $\delta$-independent estimators (conditioned on
  $S$), each of which is equal to $S$ with probability at least
  $p^*-\eps$. By Claims~\ref{cor:majority}
  and~\ref{clm:pGreaterThanHalf}, the MAP estimator of $S$ using these
  estimators equals $S$ with probability higher than $p^*$, for the
  appropriate choice of low enough $\eps,\delta$. Therefore, since
  $w$'s action $\action_w(T)$ is the MAP estimator of $S$, its
  probability of equaling $S$ is $\P{\action_w(T)=S} > p^*$ as well,
  and so $p(H) > p^*$ - contradiction.
\end{proof}

Using Theorem~\ref{thm:bounded} we prove
Theorem~\ref{thm:bounded-learning}, which is the corresponding
theorem for finite graphs:
\begin{theorem*}[\ref{thm:bounded-learning}]
  Let $\mu_0,\mu_1$ be such that for every connected, undirected graph
  $G$ there exists a random variable $\limac$ such that almost surely
  $\limac_u=\limac$ for all $u \in V$.  Then there exists a sequence
  $q(n)=q(n, \mu_0, \mu_1)$ such that $q(n) \to 1$ as $n \to \infty$,
  and $\P{\limac = \{S\}} \geq q(n)$, for any choice of undirected,
  connected graph $G$ with $n$ agents.
\end{theorem*}
\begin{proof}
  Assume the contrary. Then there exists a series of graphs $\{G_r\}$
  with $r$ agents such that $\lim_{r \to \infty}\P{\limac(G_r)=\{S\}} < 1$,
  and so also $\lim_{r \to \infty}p(G_r) < 1$.

  By the same argument of Theorem~\ref{thm:bounded} these graphs must
  all be in $\InfGraphs_d$ for some $d$, since otherwise, by
  Lemma~\ref{thm:large-out-deg-graph}, there would exist a subsequence
  of graphs $\{G_{r_d}\}$ with degree at least $d$ and $\lim_{d \to
    \infty}p(G_{r_d}) = 1$. Since $\InfGraphs_d$ is compact
  (Lemma~\ref{lemma:inf-graphs-closed}), there exists a graph $(G,u)
  \in \InfGraphs_d$ that is the limit of a subsequence of $\{(G_r,
  u_r)\}_{r=1}^\infty$.

  Since $G$ is infinite and of bounded degree, it follows by
  Theorem~\ref{thm:bounded} that $p(G) = 1$, and in particular
  $\lim_{r \to \infty}p_u(r) = 1$. As before, $p_{u_r}(r) = p_u(r)$,
  and therefore $\lim_{r \to \infty}p_{u_r}(r) = 1$. Since $p(G_r)
  \geq p_{u_r}(r)$, $\lim_{r \to \infty}p(G_r) = 1$, which is a
  contradiction.
\end{proof}

\subsection{Convergence to identical optimal action sets}
\label{sec:convergence}
In this section we prove Theorem~\ref{thm:unbounded-common-knowledge}.
\begin{theorem*}[\ref{thm:unbounded-common-knowledge}]
  Let $(\mu_0,\mu_1)$ induce non-atomic beliefs. Then there exists
  a random variable $\limac$ such that almost surely $\limac_u=\limac$
  for all $u$.
\end{theorem*}

In this section we shall assume henceforth that the distribution of
initial private beliefs is non-atomic.

\subsubsection{Previous work}
The following theorem is due to Gale and Kariv~\cite{GaleKariv:03}.
Given two agents $u$ and $w$, let $E_u^0$ denote the event that
$\action_u(t)$ equals $0$ infinitely often $E_w^1$ and the event that
$\action_w(t)$ equals $1$ infinitely often.
\begin{theorem}[Gale and Kariv]
  \label{thm:imitation}
  If agent $u$ observes agent $w$'s actions then
  \begin{equation*}
    \P{E_u^0, E_w^1} = \P{\belief_u = 1/2, E_u^0, E_w^1}.
  \end{equation*}
\end{theorem}
I.e., if agent $u$ takes action 0 infinitely often, agent $w$ takes
action 1 infinitely, and $u$ observes $w$ then $u$'s belief is $1/2$
at the limit, almost surely.
\begin{corollary}
  \label{cor:imitation}
  If agent $u$ observes agent $w$'s actions, and $w$ takes both
  actions infinitely often then $X_u=1/2$.
\end{corollary}
\begin{proof}
  Assume by contradiction that $X_u < 1/2$. Then $u$ takes
  action 0 infinitely often. Therefore Theorem~\ref{thm:imitation}
  implies that $X_u=1/2$ - contradiction.

  The case where $X_u > 1/2$ is treated similarly.
\end{proof}

\subsubsection{Limit log-likelihood ratios}
Denote
\begin{align*}
  Y_u(t) =
  \log\frac{\CondP{I_u(t)}{S=1,\bar{\action}_u(t)}}{\CondP{I_u(t)}{S=0,\bar{\action}_u(t)}}.
\end{align*}
In the next claim we show that $\llr_u(t)$, the log-likelihood ratio
inspired by $u$'s observations up to time $t$, can be written as the
sum of two terms: $\llr_u(1)=\frac{d\mu_1}{d\mu_0}(\psignal_u)$, which
is the log-likelihood ratio inspired by $u$'s private signal
$\psignal_u$, and $Y_u(t)$, which depends only on the actions of $u$
and its neighbors, and does not depend directly on $\psignal_u$.

\begin{claim}
  \label{clm:llr-decomp}
  \begin{align*}
    \llr_u(t) = \llr_u(1) + Y_u(t).
  \end{align*}
\end{claim}
\begin{proof}
  By definition we have that
  \begin{align*}
    \llr_u(t) =
    \log\frac{\CondP{S=1}{\cF_u(t)}}{\CondP{S=0}{\cF_u(t)}}
    = \log\frac{\CondP{S=1}{I_u(t),\psignal_u}}{\CondP{S=0}{I_u(t),\psignal_u}}.
  \end{align*}
  and by the law of conditional probabilities
  \begin{align*}
    \llr_u(t) &= \log\frac{\CondP{I_u(t)}{S=1,
            W_u}\CondP{W_u}{S=1}}{\CondP{I_u(t)}{S=0,W_u}\CondP{W_u}{S=0}}\\
      &=\log\frac{\CondP{I_u(t)}{S=1,
          W_u}}{\CondP{I_u(t)}{S=0,W_u}}+\llr_u(1).
  \end{align*}
  Now $I_u(t)$, the actions of the neighbors of $u$ up to time $t$,
  are a deterministic function of $\psignal(B_t(G,u))$, the private signals
  in the ball of radius $t$ around $u$, by
  Claim~\ref{clm:a-from_cf}. Conditioned on $S$ these are all
  independent, and so, from the definition of actions, these actions
  depend on $u$'s private signal $\psignal_u$ only in as much as it
  affects the actions of $u$. Hence
  \begin{align*}
    \CondP{I_u(t)}{S=s,W_u} = \CondP{I_u(t)}{S=s,\bar{\action}_u(t)},
  \end{align*}
  and therefore
  \begin{align*}
    \llr_u(t) &=\log\frac{\CondP{I_u(t)}{S=1,
        \bar{\action}_u(t)}}{\CondP{I_u(t)}{S=0,\bar{\action}_u(t)}}+\llr_u(1)\\
    &= \llr_u(1)+Y_u(t).
  \end{align*}
\end{proof}
Note that $Y_u(t)$ is a deterministic function of $I_u(t)$ and
$\bar{\action}_u(t)$.

Following our notation convention, we define $Y_u = \lim_{t \to
  \infty} Y_u(t)$. Note that this limit exists almost surely since the
limit of $\llr_u(t)$ exists almost surely. The following claim follows
directly from the definitions:
\begin{claim}
  \label{clm:y-measurable}
  $Y_u$ is measurable in $(\bar{\action}_u,I_u)$, the actions of $u$
  and its neighbors.
\end{claim}

\subsubsection{Convergence of actions}
The event that an agent takes both actions infinitely often is (almost
surely) a sufficient condition for convergence to belief $1/2$. This
follows from the fact that these actions imply that its belief takes
values both above and below $1/2$ infinitely many times. We show that
it is also (almost surely) a necessary condition. Denote by $E_u^a$
the event that $u$ takes action $a$ infinitely often.
\begin{theorem}
  \label{thm:mu-half-indecisive}
  \begin{equation*}
    \P{E_u^0 \cap E_u^1, \belief_u = 1/2} = \P{\belief_u = 1/2}.
  \end{equation*}
\end{theorem}
I.e., it a.s.\ holds that $X_u=1/2$ iff $u$ takes both actions
infinitely often.
\begin{proof}
  We'll prove the claim by showing that $\P{\neg(E_u^0 \cap E_u^1),
    \belief_u = 1/2} = 0$, or equivalently that $\P{\neg(E_u^0 \cap
    E_u^1), \llr_u = 0} = 0$ (recall that $\llr_u = \log
  \belief_u/(1-\belief_u)$ and so $\belief_u=1/2 \Leftrightarrow
  \llr_u=0$).

  Let $\bar{a}=(a(1), a(2), \ldots)$ be a sequence of actions, and
  denote by $\psignal_{-u}$ the private signals of all agents except
  $u$. Conditioning on $\psignal_{-u}$ and $S$ we can write:
  \begin{align*}
    \P{\bar{\action}_u=\bar{a}, \llr_u=0} &=
    \E{\CondP{\bar{\action}_u=\bar{a},
        \llr_u=0}{\psignal_{-u},S}} \\
    &= \E{\CondP{\bar{\action}_u=\bar{a},
        \llr_u(1)=-Y_u}{\psignal_{-u},S}}
  \end{align*}
  where the second equality follows from Claim~\ref{clm:llr-decomp}.
  Note that by Claim~\ref{clm:y-measurable} $Y_u$ is fully determined
  by $\bar{\action}_u$ and $\psignal_{-u}$. We can therefore write
  \begin{align*}
    \P{\bar{\action}_u=\bar{a}, \llr_u=0} &=
    \E{\CondP{\bar{\action}_u=\bar{a},
        \llr_u(1)=-Y_u(\psignal_{-u},\bar{a})}{\psignal_{-u},S}} \\
    &\leq
    \E{\CondP{\llr_u(1)=-Y_u(\psignal_{-u},\bar{a})}{\psignal_{-u},S}}
  \end{align*}

  Now, conditioned on $S$, the private signal $\psignal_u$ is
  distributed $\mu_S$ and is independent of $\psignal_{-u}$. Hence its
  distribution when further conditioned on $\psignal_{-u}$ is still
  $\mu_S$. Since $\llr_u(1) = \log\frac{d\mu_1}{d\mu_0}(\psignal_u)$,
  its distribution is also unaffected, and in particular is still
  non-atomic. It therefore equals $-Y_u(\psignal_{-u},\bar{a})$ with
  probability zero, and so
  \begin{align*}
    \P{\bar{\action}_u=\bar{a}, \llr_u=0} = 0.
  \end{align*}
  Since this holds for all sequences of actions $\bar{a}$, it holds in
  particular for all sequences which converge. Since there are only
  countably many such sequences, the probability that the action
  converges (i.e., $\neg(E_u^0 \cap E_u^1)$) and $\llr_u=0$ is zero,
  or
  \begin{align*}
    \P{\neg(E_u^0 \cap E_u^1), \llr_u = 0} = 0.
  \end{align*}
\end{proof}

Hence it impossible for an agent's belief to converge to $1/2$ and for
the agent to only take one action infinitely often. A direct
consequence of this, together with Thm.~\ref{thm:imitation}, is the
following corollary:
\begin{corollary}
  \label{cor:converge}
  The union of the following three events occurs with probability one:
  \begin{enumerate}
  \item $\forall u\in V:\lim_{t \to \infty} \action_u(t)=S$. Equivalently, all
    agents converge to the correct action.
  \item $\forall u\in V:\lim_{t \to \infty} \action_u(t)=1-S$. Equivalently, all
    agents converge to the wrong action.
  \item $\forall u\in V:\belief_u=1/2$, and in this case all agents take
    both actions infinitely often and hence don't converge at all.
  \end{enumerate}
\end{corollary}
\begin{proof}
  Consider first the case that there exists a vertex $u$ such that $u$
  takes both actions infinitely often. Let $w$ be a vertex that
  observes $u$. Then by Corollary~\ref{cor:imitation} we have that
  $X_w=1/2$, and by Theorem~\ref{thm:mu-half-indecisive} $w$ also
  takes both actions infinitely often. Continuing by induction and
  using the fact that the graph is strongly connected we obtain the
  third case that none of the agents converge and $X_u=1/2$ for all
  $u$.

  It remains to consider the case that all agents' actions converge to
  either 0 or 1. Using strong connectivity, to prove the theorem it
  suffices to show that it cannot be the case that $w$ observes $u$
  and they converge to different actions. In this case, by
  Corollary~\ref{cor:imitation} we have that $X_w=1/2$, and then by
  Theorem~\ref{thm:mu-half-indecisive} agent $w$'s actions do not
  converge - contradiction.
\end{proof}

Theorem~\ref{thm:unbounded-common-knowledge} is an easy consequence of
this theorem.  Recall that $\limac_u = \{1\}$ when $\belief_u > 1/2$,
$\limac_u = \{0\}$ when $\belief_u < 1/2$ and $\limac_u = \{0,1\}$
when $\belief_u=1/2$.
\begin{theorem*}[\ref{thm:unbounded-common-knowledge}]
  Let $(\mu_0,\mu_1)$ induce non-atomic beliefs. Then there exists
  a random variable $\limac$ such that almost surely $\limac_u=\limac$
  for all $u$.
\end{theorem*}
\begin{proof}
  Fix an agent $v$. When $\belief_v < 1/2$ (resp.\ $\belief_v > 1/2$)
  then the first (resp.\ second) case of corollary~\ref{cor:converge}
  occurs and $\limac=\{0\}$ (resp.\ $\limac=\{1\}$). Likewise when
  $\belief_v=1/2$ then the third case occurs, $\belief_u = 1/2$ for
  all $u \in V$ and $\limac_u = \{0,1\}$ for all $u \in V$.
\end{proof}

\subsection{Extension to $L$-locally connected graphs}
\label{sec:locally-connected}
The main result of this article, Theorem~\ref{thm:bounded-learning},
is a statement about undirected graphs. We can extend the proof to a
larger family of graphs, namely, $L$-locally connected graphs.
\begin{definition}
  Let $G=(V,E)$ be a directed graph. $G$ is $L$-locally strongly
  connected if, for each $(u,w) \in E$, there exists a path in $G$ of
  length at most $L$ from $w$ to $u$.
\end{definition}

Theorem~\ref{thm:bounded-learning} can be extended as follows.
\begin{theorem}
  \label{thm:l-locally}
  Fix $L$, a positive integer. Let $\mu_0,\mu_1$ be such that for
  every strongly connected, directed graph $G$ there exists a random
  variable $\limac$ such that almost surely $\limac_u=\limac$ for all
  $u \in V$.  Then there exists a sequence $q(n)=q(n, \mu_0, \mu_1)$
  such that $q(n) \to 1$ as $n \to \infty$, and $\P{\limac = \{S\}}
  \geq q(n)$, for any choice of $L$-locally strongly connected graph
  $G$ with $n$ agents.
\end{theorem}

The proof of Theorem~\ref{thm:l-locally} is essentially identical to
the proof of Theorem~\ref{thm:bounded-learning}. The latter is a
consequence of Theorem~\ref{thm:bounded}, which shows learning in
bounded degree infinite graphs, and of
Lemma~\ref{thm:large-out-deg-graph}, which implies asymptotic learning
for sequences of graphs with diverging maximal degree.

Note first that the set of $L$-locally strongly connected rooted
graphs with degrees bounded by $d$ is compact. Hence the proof of
Theorem~\ref{thm:bounded} can be used as is in the $L$-locally
strongly connected setup.

In order to apply Lemma~\ref{thm:large-out-deg-graph} in this setup,
we need to show that when in-degrees diverge then so do
out-degrees. For this note that if $(u,v)$ is a directed edge then $u$
is in the (directed) ball of radius $L$ around $v$. Hence, if there
exists a vertex $v$ with in-degree $D$ then in the ball of radius $L$
around it there are at least $D$ vertices. On the other hand, if the
out-degree is bounded by $d$, then the number of vertices in this ball
is at most $L\cdot d^L$. Therefore, $d \to \infty$ as $D \to \infty$.

\appendix

\section{Example of Non-atomic private beliefs leading to
  non-learning}
\label{app:example-atomic}

We sketch an example in which private beliefs are atomic and
asymptotic learning does not occur.

\begin{example}
  \label{example:non-learning}
  Let the graph $G$ be the undirected chain of length $n$, so that
  $V=\{1, \ldots, n\}$ and $(u,v)$ is an edge if $|u-v| = 1$. Let the
  private signals be bits that are each independently equal to $S$
  with probability $2/3$. We choose here the tie breaking rule under
  which agents defer to their original signals\footnote{We conjecture
    that changing the tie-breaking rule does not produce asymptotic
    learning, even for randomized tie-breaking.}.
\end{example}

We leave the following claim as an exercise to the reader.
\begin{claim}
  If an agent $u$ has at least one neighbor with the same private
  signal (i.e., $\psignal_u=\psignal_v$ for $v$ a neighbor of $u$)
  then $u$ will always take the same action $\action_u(t) =
  \psignal_u$.
\end{claim}
Since this happens with probability that is independent of $n$, with
probability bounded away from zero an agent will always take the wrong
action, and so asymptotic learning does not occur. It is also clear
that optimal action sets do not become common knowledge, and these
fact are indeed related.

\bibliographystyle{spmpsci}
\bibliography{all}
\end{document}